\documentclass[reqno,12pt]{amsart}
\usepackage{amssymb}
\usepackage{enumerate}
\usepackage[all,matrix,arrow]{xy}
\usepackage{comment}
\usepackage{color}

\newcommand{\PP}{\mathbb{P}}
\newcommand{\C}{\mathbb{C}}
\newcommand{\Pl}{\PP^2}
\newcommand{\cD}{\mathcal{D}}
\newcommand{\cL}{\mathcal{L}}
\newcommand{\cO}{\mathcal{O}}

\newcommand{\F}{\mathbb{F}}
\newcommand{\Q}{\mathbb{Q}}
\newcommand{\eps}{\varepsilon}
\newcommand{\rto}{\dasharrow}
\newcommand{\infnear}[1][]{>^{#1}}
\newcommand{\prox}{\to}
\newcommand{\satel}{\odot}
\newcommand{\notsatel}{\hbox{$\mspace{5mu}\not\mspace{-5mu}\odot\mspace{6mu}$}}

\DeclareMathOperator{\Cdeg}{cdeg}

\DeclareMathOperator{\Pic}{Pic}
\DeclareMathOperator{\elm}{elm}
\DeclareMathOperator{\ad}{ad}
\DeclareMathOperator{\mult}{mult}
\DeclareMathOperator{\vdeg}{vdeg}
\DeclareMathOperator{\id}{id}

\newtheorem{theorem}{Theorem}[section]
\newtheorem{proposition}[theorem]{Proposition}
\newtheorem{lemma}[theorem]{Lemma}
\newtheorem{corollary}[theorem]{Corollary}

\theoremstyle{definition}

\newtheorem{examples}[theorem]{Examples}
\newtheorem{remark}[theorem]{Remark}
\newtheorem{remarks}[theorem]{Remarks}

\title[Birational classification of curves on rational surfaces]%
{Birational classification of curves\\on rational surfaces}

\author{Alberto Calabri \and Ciro Ciliberto}

\email{calabri@dmsa.unipd.it}
\curraddr{Dipartimento di Metodi e Modelli Matematici per le Scienze Applicate,
Universit\`a degli Studi di Padova,
Via Trieste 63, 35121 Padova, Italy, phone: +39-049-827-1318, fax: +39-049-827-1333}

\email{cilibert@mat.uniroma2.it}
\curraddr{Dipartimento di Matematica, Universit\`a degli Studi di Roma ``Tor Vergata'',
Via della Ricerca Scientifica, 00133 Roma, Italy, phone: +39-06-7259-4684,
fax: +39-06-7259-4699}

\thanks{Mathematics Subject Classification (2000): 14E05;
(Secondary) 14J26, 14H50, 14E07, 14E30.\\
The authors are member of G.N.S.A.G.A.\ at I.N.d.A.M.\ ``Francesco Severi''}

\begin{document}

\begin{abstract}
In this paper we consider the birational classification of pairs $(S,\cL)$, with $S$ a rational surfaces and 
$\cL$ a linear system on $S$. We give a classification theorem for such pairs  and we determine, for each irreducible plane curve $B$, its \emph{Cremona minimal} models, i.e.\ those plane curves which are equivalent to $B$ via a Cremona transformation, and have minimal degree under this condition. 
\end{abstract}

\maketitle


\tableofcontents


\section*{Introduction}
Let $B$ be an irreducible curve in the complex projective plane $\PP^2$. A natural question is to look for its  \emph{Cremona minimal} models, i.e.\ for those curves which are equivalent to $B$ via a Cremona transformation of $\PP^2$, and have minimal degree under this condition. More generally, the same question can be asked for positive dimensional linear systems $\cL$ of plane curves.

This is a classical problem, which goes back to the very beginnings of birational geometry in the second half of XIX century,  the main characters being L.~Cremona and his school  and M.~Noether. This question has been later considered by several algebraic geometers for many decades till the 1940's.  Indeed,  a long series of papers by various classical Authors was devoted to trying to solve this problem at least for planar linear systems of curves of low genera. Giving here a full account, with complete references, of all these attempts, some of them affected by serious gaps, 
would turn these few lines into an historical work rather than  an introduction as it is.
So we will briefly
address the interested reader to  Coolidge's book \cite{Coolidge}, whose first edition appeared in 1931, which contains an exposition of the classical results and a detailed bibliography on the subject. Another beautiful classical reference is Enriques--Conforto's book \cite{Conforto}, which contains the most advanced classical treatment of the subject, as well as interesting historical notes. We cannot resist however mentioning, among all the others, Castelnuovo's contributions \cite{Ca, Ca2, Cast}, where
adjunction theory is fully exploited in a form that, after Mori's epochal work, we call ``running a minimal model program'' driven by a given divisor on a surface, i.e.\ a $\sharp$--minimal model program in M. Reid's terminology \cite{Re}. 

In modern times, the question has been considered again starting from the 1960's by a few Authors, among which, in chronological order,  we mention Nagata \cite{Nagata1, Nagata2},  Kumar and Murthy \cite{KumarMurthy}, Dicks \cite{Dicks}, Reid \cite{Re},  Iitaka \cite{Iitaka1, Iitaka2} and Matsuda \cite{Matsuda}. 

In a nutshell, the problem boils down to consider the birational classification of pairs $(S,\cL)$, with $S$ a rational surface and $\cL$ a linear system on it. Taking this viewpoint, the most appropriate tool available today for attacking the question consists in using the machinery of Mori's program, in its log--version. This is essentially Dick's  and Reid's viewpoint,
and this is basically what we also do here, though we rather use Iitaka's terminology and the more classical approach via adjunction theory \emph{\'a la} Castelnuovo.
Part of our main results is stated in the following two theorems (for more detailed statements, see  Theorems \ref{thm1} and  \ref{thm2}).

\begin{theorem}\label{thmA} 
Let $(S,C)$ be a pair with $S$ rational and $C$ smooth and irreducible,
which is not birationally equivalent to $(\PP^2, L)$, where $L$ is  a line. Let $m$ be the maximum integer such that
$\vert C+mK_S\vert$ is not empty and let $\alpha$ be the dimension of this system. 
Then $(S,C)$ is birationally equivalent to one of the following pairs:
\begin{enumerate} [(i)]

\item  $(\Pl,D)$, where
$D$ is a plane curve of degree $d\geqslant 3$ with points of multiplicity strictly smaller than $m-1$ where $m=[\frac d3]$;

\item  $(\F_n,D)$,  where $\F_n=\PP(\mathcal O_{\PP^1}\oplus \mathcal O_{\PP^1}(-n))$,
\[
D\in \left\vert (2m+\epsilon)E
  +\left((2+n)m+\dfrac {\alpha+\epsilon (n-1)} {1+\epsilon}\right)F\right\vert,
\]
$\epsilon\in\{0,1\}$,
$E$ is a curve with $E^2=-n$ and $F$ is a ruling of $\F_n$,
and: 
\begin{itemize}
\item $D$ is irreducible with points of multiplicity  at most $m$;
\item if $\epsilon=0$ and $n>0$, the singular points of $D$ of multiplicity $m$ lie on $E$ and $n$ is minimal under this condition.
\end{itemize}
\end{enumerate}

These pairs may be birationally equivalent only if:
\begin{itemize}
\item $\epsilon=0$ and $D$ has at least two points of multiplicity $m$;
\item $n=\epsilon=1$, $\alpha=0$ [resp. $\alpha=2$] and $D$ with at least three [resp. two] points of multiplicity $m$.
\end{itemize}
\end{theorem}

\begin{theorem}\label{thmB}  Let $C$ be an irreducible plane curve. 
Then a Cremona minimal model $B$ of $C$ is of one of the following types:

\begin{enumerate}[(i)]

\item $B$ is a line;

\item if $B$ has degree $d$ and points of multiplicities $m_1,\ldots, m_r$ with  
$m_1\geqslant \cdots\geqslant m_r$, then $d\geqslant m_1+m_2+m_3$;

\item $B$ is a curve of degree $d$ with a point $p$ of maximal multiplicity $m_0$ and all points of multiplicity $\mu\geqslant ({d-m_0})/2$ are infinitely near  to $p$. \end{enumerate}
\end{theorem}

The structure of the Cremona minimal curves in (iii) can be well specified, as in the statement of Theorem \ref{thm2}. They
 are obtained from the pairs $(\F_n,D)$ in Theorem \ref{thmA} with a suitable process which is explained in detail in \S\ref{S:sing}.  The curves in (ii) may be birationally, and not projectively, equivalent only if $d= m_1+m_2+m_3$. This is a consequence of a result, asserted by several classical authors and proved by G.~Jung in \cite{Jung1}, to the effect that a linear system of plane curves of degree $d$ with points of multiplicities $m_1\geqslant m_2\geqslant m_3\geqslant \ldots$, is Cremona minimal as soon as $d\geqslant m_1+m_2+m_3$ and if $d>m_1+m_2+m_3$ any Cremona minimal linear system birational to it, is projectively equivalent to it. We give a short and easy proof of this in \S \ref{sec:Equiv}. As for type (iii), it is possible to have several different Cremona minimal, non projectively equivalent models, with different  multiplicities. 
 
In our view, these results completely solve the classification problem, though the difficulty remains, given a specific curve $B$, of resolving its singularities and determining the structure of its subsequent adjoint linear systems. 

The paper is organized as follows. In \S \ref{sec:prel} we fix notation and recall a few facts about infinitely near points and linear systems.  In \S \ref{sec:Equiv}  we prove some basic results on Cremona minimality and the aforementioned theorem of Jung. 
In \S \ref{subsec:-1cicli} we recall a few properties of $(-1)$--cycles, i.e.\ those effective divisors $C$ on a smooth surface $S$ which are contracted to a smooth point by a birational morphism, which is an isomorphism on the complement of $C$ on $S$. This is used in \S \ref{sec:adjprop}, where we prove some essential nefness results on effective adjoint linear systems which boil down to computing their Zariski decomposition, and we recall how adjoint systems behave under birational maps. 
In \S \ref{sec:Iitaka} we introduce Iitaka's $\sharp$--models, and define the equally useful $\flat$--models and $\natural$--models. We devote \S \ref{subsec:moving} to briefly recalling the birational  classification of pairs $(S,\cL)$, with $\cL$ a nef linear system of positive dimension of curves of arithmetic genus 0, and, accordingly, the Cremona classification of planar positive dimensional  linear systems of rational curves: this is the first stone of our classification. 
The  following \S\S \ref{admissible} and \ref{S:sing} are devoted to introducing and constructing planar linear systems enjoying the Cremona minimality property. As indicated in the statement of Theorem \ref{thmB}, their main feature is that the maximal singularities of their general curve are nestled in a rather complicated way infinitely near to the point of highest multiplicity. In \S \ref{S:theorem} we prove the announced classification results for pairs $(S,C)$ (and for plane curves $B$), by subdividing them into four main classes: the line, the del Pezzo, the ruled and the big case, according to the behaviour of the last effective adjoint system $\vert C+mK_S\vert$. An Appendix is devoted to quickly proving the famous Noether--Castelnuovo's Theorem on the generation of the Cremona group via linear and quadratic transformations:
this is done exploiting the concept of \emph{simplicity} of a curve (see \cite{Alex, Chisini, calabri2}, which we effectively use in 
\S\S \ref{admissible} and \ref{S:sing} to prove our Cremona minimality results. 

In \S \ref{S:applications} we present a few applications, the most relevant of which is the proof of  theorem originally stated by de Franchis \cite{Defranchis}, which classifies, up to Cremona transnformations, planar linear systems of positive dimension of curves of genus 2. De Franchis original proof is affected, as well as all papers on the subject appeared before 1901, by a criticism raised by C.~Segre to Noether's original proof of Noether--Castelnuovo's Theorem  \cite{Segre1}.

When the present research was completed, M.~Mella kindly brought to our attention his pre-print  \cite{MellaPolastri} in collaboration with E.~Polastri, in which, among other interesting things, similar results are contained, though the classification there is less fine than the one we produce here.

\section{Preliminaries}\label{sec:prel}

\subsection{Notation and conventions} In this paper we will work over $\mathbb C$. 

Let $S$ be a smooth, irreducible, projective surface, simply  called a \emph{surface} in the sequel.
We will use standard notation in surface theory, i.e.\ $K=K_S$ will denote a canonical divisor, $q=q(S)$  the \emph{irregularity} of $S$,  $\kappa=\kappa(S)$  the \emph{Kodaira dimension}, etc. The linear equivalence of divisors will be denoted by $\equiv$.

Let $D$ be a divisor on $S$. As usual $\cO_S(D)$ will be the related invertible sheaf. We will denote by $0$ the zero divisor.  If $D$ is effective, it will be called a \emph{curve} and $p_a(D)$ will denote its \emph{arithmetic genus}.

If $A$ and $B$ are divisors on $S$, we will use the notation $A\geqslant B$ if $A-B$ is effective, $A>B$ if $A\geqslant B$ and $A\neq B$.
 Recall that a divisor $D>0$ is said to be \emph{numerically connected}
if $D=A+B$, with $A,B>0$, implies $A\cdot B>0$.
A divisor $D$ is \emph{nef} if for any curve $C$ one has $D\cdot C\geqslant 0$. We may sometimes consider 
$\Q$--divisors.

If $C$ is a smooth, irreducible curve with $p_a(C)=0$ and $C^2=-k<0$, we will say that $C$ is a \emph{$(-k)$-curve}. By Castelnuovo's Theorem, a $(-1)$-curve is the exceptional divisor of a blow-up.

If $f\colon S\to S'$ is a morphism and $D, D'$ are divisors on $S,S'$ respectively, it makes sense  to consider,  the \emph{image} $f_*(D)$ of $D$ on $S'$, the \emph{total transform} $f^*(D')$ and the \emph {strict}  or  \emph{proper transform} of $D'$ on $S$ 
(see \cite{Hartshorne}, p. \ 110 and p.\ 425, and  \cite{Mats}. p. 121).

\subsection{Infinitely near points} \label{ssec:inp} Here we briefly recall some basic facts and terminology about infinitely near points, which will be commonly used in the sequel
(cf., e.g., \cite{Alberich} and \cite{Enriques}).

Let $S$ and $S'$ be surfaces.
Any birational morphism $\sigma\colon S'\to S$
is the composition of a certain number, say $n$,
blowing-ups $\sigma_i\colon S_{i}\to S_{i-1}$
at a point $p_i\in S_{i-1}$, $i=1,\ldots,n$:
\begin{equation}\label{eq:seq1}
\sigma \colon S'=S_{n}
  \xrightarrow{\,\sigma_n\,}  S_{n-1}
  \xrightarrow{\,\sigma_{n-1}\,}
  \cdots
  \xrightarrow{\,\sigma_{2}\,}  S_{1}
  \xrightarrow{\,\sigma_{1}\,}  S_0=S.
\end{equation}

Let $p\in S$ be a point.
One says that $q$ is an \emph{infinitely near point to $p$ of order $n$},
and we write $q\infnear[n] p$, if there exists
a birational morphism $\sigma\colon S'\to S$ as in \eqref{eq:seq1},
such that $p_1=p$, $\sigma_i(p_{i+1})=p_{i}$, $i=1,\ldots,n-1$,
and $q\in Z_n=\sigma_{n}^{-1}(p_{n-1})$.
For each $i=1,\ldots,n$, let $E_i=\sigma_{i}^{-1}(p_{i})\subset S_i$
be the exceptional curve of $\sigma_i$ and let $E'_i$ be the strict transform of $E_i$ on $S'$.
If $i>j$, let $\sigma_{i,j}$ be the morphism $S_i\to S_j$.
For each $i=1,\ldots,n-1$, set $Z_i=\sigma_{n,i+1}^*(E_{i})$.
According to the terminology introduced in \S \ref{subsec:-1cicli} below,
$Z_1,\ldots,Z_{n-1},Z_n=E_n$ are \emph{$(-1)$-cycles}
generating $\Pic(S')$ over $\Pic(S)$.

One says that $q$ is \emph{proximate to $p$}, and we write $q\prox p$,
if either $q\infnear[1] p$ 
or $q\infnear[n] p$ with $n>1$ and $q$ lies on the strict transform $E'_1$ of $E_1$ on $S'$.
In the latter case, one says that $q$ is \emph{satellite to $p$},
and we write $q \satel p$.
This may happen only if $p_i$ lies on the strict transform
of $E_1$ on $S_{i-1}$, for each $i=2,\ldots,n$.

Also $E'_1,\ldots,E'_{n-1},E'_n=E_n$ generate
$\Pic(S')$ over $\Pic(S)$ and $E'_i=Z_i-\sum_{j}q_{ij}Z_j$,
where $q_{ij}=1$ if $p_j\prox p_i$, and $q_{ij}=0$ otherwise.

In this paper, we will usually refer to \emph{points} on a surfaces $S$ including
infinitely near ones.
We will say that a point $p$ is \emph{proper},
and we will write $p\in S$,
if $p$ is not infinitely near to any point of $S$.
An infinitely near point is called \emph{free} if it is not satellite (to any point).

Let now $C'$ be a curve on $S'$ and $C=\sigma_*(C')$.
Then
$
C'=\sigma^*(C)-\sum_{i=1}^n m_iZ_i,
$
where $m_1,\ldots,m_n$ are integers.
If $C$ is a curve, i.e.\ if $C'$ is not contracted by $\sigma$,
one says that $m_i$, $i=1,\ldots,n$, is the \emph{(virtual) multiplicity} of $C$
at the point $p_i$.
If no component of $C'$ is contracted by $\sigma$,
then, for each $i=1,\ldots,n$, one has $C'\cdot E'_i\geqslant0$, which is equivalent to
\begin{equation}\label{eq:prox}
m_i\geqslant \sum_{j\colon p_j\prox p_i} m_j.
\end{equation}
whichis classically known as the
\emph{proximity inequality at $p_i$}.

In any case, we may uniquely write $C'=D'+\sum_{i=0}^{n} h_iE'_{i}$,
with $D'$ effective and $D'\cdot E'_{i}\geqslant 0$
(i.e.\ the proximity inequality at $p_i$ hold), for all $i=1,\ldots,n$. 
We say that $D'$ is \emph{pure}.
One has $\sigma_*(D')=\sigma_*(C')=C$.
Then
$
D'\equiv\sigma^*(C)-\sum_{i=1}^n \tilde m_iZ_i,
$
where $\tilde m_1,\ldots, \tilde m_n$ are non-negative integers, called the
\emph{effective multiplicities} of $C$ at $p_1,\ldots,p_{n}$. 
One may compute the $\tilde m_i$'s from the $m_i$ by a well--known algorithm
called Enriques'  \emph{unloading principle} (cf.\ \cite{Enriques}).

\subsection{Linear systems} Let $S$ be a surface and $D$ a divisor on it.
As usual we denote by  $\vert D\vert$ the \emph{complete linear system} associated to $D$, i.e.\ $\PP(H^0(S, \cO_S(D)))$.

A \emph{linear system} $\cL\subset \vert D\vert$ corresponds to a vector subspace $V\subset  H^0(S, \cO_S(D))$.  Recall that $\cL$ determines a rational map $\phi_{\cL}\colon S \rto \PP^r=\PP(V^*)$, $r=\dim(\cL)$. Two linear systems define the same map, if they differ by divisorial fixed components. If $\cL$ has no divisorial fixed component, then $\phi_{\cL}$ is a morphism off the base locus of $\cL$. A linear system $\cL$  will be called  \emph{irreducible} if its general curve is irreducible.

If $\cL\subset \vert D\vert$ and $\cL'\subset \vert D'\vert$ are linear systems, corresponding to  $V\subset  H^0(S, \cO_S(D))$ and $V'\subset  H^0(S, \cO_S(D'))$, we will denote by
$\cL\cdot \cL'$ the intersection number $C\cdot C'$, where $C\in\cL$ and $C'\in\cL'$,
and by $\cL+\cL'$ the linear system corresponding to the image of $V\otimes V'\subset  H^0(S, \cO_S(D))\otimes  H^0(S, \cO_S(D'))$ in $ H^0(S, \cO_S(D+D'))$ via the natural multiplication map. Thus we may consider the multiple linear system $i\cL$ for all integers $i\geqslant 1$.

Let $\phi\colon S\rto S'$ be a dominant rational map, let $\cL$ be a linear system on $S$ and let $C\in \cL$ be its general curve. We will denote by $\phi_*(\cL)$ the \emph{proper image} of $\cL$, which is the linear system on $S'$ whose general curve  is the closure of the images of $\phi(x)$ where $x$ varies among the general points of all irreducible components of $C$.

One also defines the \emph{total transform} $\phi_!(\cL)$ of $\cL$ via $\phi$
which is defined as follows.
There is a commutative diagram
\[
\xymatrix{ & X \ar[dl]_{f} \ar[dr]^{f'} \\
S\ar@{-->}[rr]^{\phi} & & S'}
\]
where $X$ is a surface and $f,f'$ are birational morphisms.
Then $\phi_!(\cL)=f'_*(f^*(\cL))$.
Note that $\phi_*(\cL)$ is a subsystem of $\phi_!(\cL)$ and
the two differ for base components which are images, via $f'$,
of exceptional divisors of $f$.

Consider a sequence of blowing-ups as in \eqref{eq:seq1} and let $\cL'$ be a linear system on $S'$. 
One can define, as in \S \ref{ssec:inp}, virtual and effective multiplicites of $\cL$ at the points $p_1,\ldots,p_{n}$ which one blows up. 

Given $p_1,\ldots,p_n$ proper or infinitely near points of $\Pl$,
given positive integers $m_1,\ldots,m_r$ for each one of them,
we denote by 
\[
\cL(d;m_1,m_2,\ldots,m_r)
\]
the linear system $\cL\subseteq \vert \cO_{\PP^2}(d)\vert $
of curves of degree $\deg(\cL)=d$ with \emph{assigned} base points
$p_1,\ldots,p_n$, with \emph{assigned}, or \emph{virtual, multiplicities}
$m_1,\ldots,m_n$.
It may happen that the system $\cL$ has further base points,
fixed components and higher multiplicities than the assigned ones.
The last phenomenon means exceptional divisors as fixed components,
if one looks at the system on the blow-up of $\Pl$ at the assigned base points.

We will usually assume that $m_1\geqslant m_2\geqslant \cdots\geqslant m_r$.
We will use the exponential notation $m_i^{e_i}$ in case of $e_i$ points
of multiplicity $m_i$.

In case we want to specify that $e_i$ points of multiplicity
$m_i$, $i=1,\ldots,l$, are infinitely near 
[resp.\ proximate]
to a point of multiplicity $m$,
we will write
\begin{align*}
&
\cL(\ldots,(m,\{m_1^{e_1},m_2^{e_2},\ldots,m_l^{e_l}\}),\ldots),
&&
[\text{resp.\ }
\cL(\ldots,(m,[m_1^{e_1},m_2^{e_2},\ldots,m_l^{e_l}]),\ldots)].
\end{align*}

Recall that the \emph{virtual dimension} of the linear system
$\cL=\cL(d;m_1,\ldots,m_r)$ is
$v(\cL)={d(d+3)}/2-\sum_{i=1}^r {m_i(m_i+1)}/2.$
One has
$\dim(\cL)\geqslant \max\{v(\cL),-1\}$
and, if the equality holds, the system is called \emph{non-special}.

\subsection{Cremona transformations}

A linear system $\cL$ of plane curves, with no divisorial fixed components,
is called a \emph{net} if $\dim (\cL)=2$.
If, in addition, the map $\phi_{\cL}\colon \Pl \rto \Pl$
is a \emph{Cremona transformation}, 
i.e.\ it is birational, the net is called \emph{homaloidal}.
In that case, the general curve $C$ of $\cL$ is irreducible and rational.
If $d$ is the degree of $\cL$ one says that $\phi_{\cL}$ has \emph{degree} $d$.
Cremona transformations of degree 1 are \emph{projective} or 
\emph{linear} transformations.

An irreducible net $\Lambda$ of type $\cL(\delta;\delta-1,1^{2\delta-2})$, $\delta\geqslant2$,
is homaloidal and the corresponding birational map 
$\phi_\Lambda\colon\Pl\rto\Pl$ is called
a \emph{de Jonqui\`eres} transformation of degree $\delta$
\emph{centered} at the base points of $\Lambda$.
If $\delta=2$, the map $\phi_\Lambda$ is a \emph{quadratic} transformation.

Any homaloidal net $\cL(\delta;\alpha_0,\alpha_1,\ldots,\alpha_r)$
is such that
\begin{align}\label{eq:invariantsLambda}
 &  \delta^2-1=\sum_{i=0}^{r}\alpha_i^2,
&&  3(\delta-1)=\sum_{i=0}^{r}\alpha_i. 
\end{align}


Recall the famous:

\begin{theorem}[Noether-Castenuovo] \label{thm:NC}
Every Cremona transformation of the plane is the composition of
finitely many linear and quadratic transformations.
\end{theorem}

In Appendix \ref{app:proof} we give a proof of this theorem
by induction on the \emph{simplicity} of homaloidal nets,
which we now introduce and will use later.

Let $\cL=\cL(d;m_0,m_1,\ldots, m_r)$ be a planar linear system
with $d\geqslant m_0\geqslant m_1\geqslant \ldots \geqslant m_r\geqslant 1$,
and let $p_i$ be the point of multiplicity $m_i$ of $\cL$
(if there is no $m_i$, let $r=-1$).
We set $m_{r+1}=0$ and $m_{-1}=\infty$.
The \emph{simplicity} of $\cL$ is the triplet $(k_\cL,h_\cL,s_\cL)$
of integers defined as follows:
\begin{align}\label{eq:simplicity}
 & k_\cL=d-m_0,
&& m_{h_\cL} >\frac{k_\cL}{2}\geqslant m_{h_\cL+1},
&& s_\cL=\sharp\,\{p_i \mid 1\leqslant i \leqslant h \text{ and } p_i \satel p_0 \}.
\end{align}
One says that $\cL'$ is \emph{simpler} than $\cL$ if
the simplicity of $\cL'$ is lexicographically smaller than the one of $\cL$.
The simplicity of a Cremona transformation $\phi\colon\PP^2\rto\PP^2$
is the one of the homaloidal net defining $\phi$.
A de Jonqui\`eres transformation of degree $\delta$
has simplicity $(1,2\delta-2,s)$, with $s\leqslant \delta-1$.

\section{Birational equivalence of pairs and Cremona minimality}\label{sec:Equiv}


Let  $(S,\cL)$ be a pair with $S$ a surface and $\cL$ a linear system on it. 
If $(S,\cL)$, $(S',\cL')$ are two such pairs, we say that they are \emph{birationally equivalent}, and we write  $(S,\cL)\sim (S',\cL')$, if there is a birational map $\phi\colon S\rto S'$ such that
$\cL'=\phi_*(\cL)$, $\cL=\phi^{-1}_*(\cL')$,
and $\phi$ is not constant on each irreducible component of the divisorial part of the base locus of $\cL$.
If moreover $\phi\colon S\to S'$ is an isomorphism,
the pairs $(S,\cL)$, $(S',\cL')$ are called \emph{isomorphic}. 
Birational equivalence is an equivalence relation. Any pair
in a class is called a \emph{model} of the class.

If $(S,\cL)\sim (S',\cL')$, then $\dim(\cL)=\dim(\cL')$. If $\dim(\cL)=0$, we have the notion of birational equivalence of pairs $(S,C)$, where $C$ is a curve on $S$.  

If  $(S,\cL)\sim (S',\cL')$, then the images of $\phi_{\cL}$ and  $\phi_{\cL'}$ are projectively equivalent.  The converse is also true,
provided $\phi_{\cL}$ and  $\phi_{\cL'}$ are birational to their images. 

Given a pair $(S,\cL)$, with $S$ rational, one can consider
all models of  $(S,\cL)$ of the form $(\Pl,\cL')$.
The ones with minimal $\deg(\cL')$ will be called
\emph{Cremona minimal} and $\deg(\cL')$ will be called
the \emph{Cremona degree} $\Cdeg(\cL)$ of $\cL$.
This definition includes the one of Cremona degree  and
Cremona minimality for curves. 
It is clear that $\Cdeg(\cL)\geqslant \Cdeg(C)$,
where $C$ is the general member of $\cL$.
As we shall see, it may happen that strict inequality holds.

If $(S,\cL)$ is birationally equivalent to $(\Pl,\cL(d;m_1^{e_1},\ldots,m_r^{e_r}))$,
we say that $\cL$ is of \emph{Cremona type} $(d;m_1^{e_1},\ldots,m_r^{e_r})$.
If $d=\Cdeg(\cL)$, we say that $\cL$ is of \emph{minimal Cremona type}
$(d;m_1^{e_1},\ldots,m_r^{e_r})$.

Consider a linear system $\cL=\cL(d;m_1,\ldots,m_r)$
and a Cremona transformation $\phi_\Lambda\colon\Pl\rto\Pl$
determined by the homaloidal net $\Lambda=\cL(\delta;\alpha_1,\ldots,\alpha_s)$.
At the cost of taking some $\alpha_i=0$,
we may assume that $s\geqslant r$.
Moreover we assume that the assigned base points of $\cL$
coincide with the points of multiplicities $\alpha_1,\ldots,\alpha_r$ of $\Lambda$.
One has
\begin{equation}\label{eq:degphi}
\deg(\phi_*(\cL))\leqslant d\delta-\sum_{i=1}^r \alpha_i m_i
\end{equation}
and the equality holds if $\cL$ has no base points off the assigned ones
and its effective multiplicities equal the assigned ones.
We denote by $\vdeg_\phi(\cL)$ the right-hand side of \eqref{eq:degphi}
and call it the \emph{virtual degree} of $\phi_*(\cL)$.

The following lemma gives useful criteria for Cremona minimality.

\begin{lemma}\label{lem:criterion} Planar linear systems $\cL$ of the following types
are Cremona minimal:
\begin{enumerate}[$(i)$]

\item $\cL(d;n,m)$, with $d\geqslant n+m$ and $n\geqslant m\geqslant 0$;

\item $\cL(d;(m,[m_1,\ldots,m_h]))$,
where $h\geqslant1$, $m_1\geqslant m_2\geqslant \cdots \geqslant m_h\geqslant 1$, $m\geqslant m_1+\cdots+m_h$,
$d\geqslant m+m_1$, and the point $p_i$, $i=1,\ldots,h$, of multiplicity $m_i$
is infinitely near of order 1 to the point $p$ of multiplicity $m$.
\end{enumerate}
Moreover a Cremona transformation $\phi\colon\PP^2\rto\PP^2$
such that $\cL'=\phi_*(\cL)$ has degree $d$ is linear except:
\begin{enumerate}[$\ \bullet\ $]

\item  in case (i), $d=n+m$, in which
$\phi$ may be a de Jonqui\`eres transformation of degree  $\delta\geqslant2$
centered at the base points of $\cL$, and $\delta=2$ unless $m=0$, 
i.e.\  unless $\cL$ is composed of a pencil of lines;

\item  in case (ii), in which
$\phi$ may be a de Jonqui\`eres transformation of degree $\delta\leqslant h'+1$,
where $h'$ is the maximum such that $d-m=m_1=\cdots=m_{h'}$, and $\phi$ is
centered at $p$ and at $\delta-1$ points among $p_1,\ldots,p_{h'}$.
\end{enumerate}
In all cases $\cL'$ is of the same Cremona type as $\cL$. 
\end{lemma}

\begin{proof} $(i)$
One sees that $\cL$ has no unassigned base point.
Suppose the homaloidal net corresponding to $\phi$ 
has degree $\delta$ and multiplicities $\alpha,\beta$ at the points of multiplicities $n,m$ of $\cL$.
We may assume that $\delta>1$, $\alpha\geqslant \beta$ and $\delta\geqslant\alpha+\beta$ and $\delta-1\geqslant \alpha$.
Then $\phi$ maps $\cL$ to a linear system $\cL'$ of degree
\[
\delta d-\alpha n-\beta m\geqslant \delta d -\alpha(n+m)\geqslant (\delta-\alpha)d
\geqslant d.
\]
If the equality holds, then either $\beta=m=0$, $n=d$ and $\alpha=\delta-1$,
or $\alpha=\beta=\delta-1=1$ and $d=n+m$,
proving the assertion in case $(i)$.

\noindent
$(ii)$
By blowing up $p$, working on $\F_1$,
and applying Bertini's Theorem,
one sees that $\cL$ has no unassigned base point
and its general curve is irreducible,
with multiplicities at the base points equal to the assigned ones.
Suppose the homaloidal net corresponding to $\phi$ 
has degree $\delta>1$ and multiplicities $\alpha$ and $\alpha_i$
respectively at $p$ and at $p_i$, $i=1,\ldots,h$.
Then $\delta-1\geqslant \alpha\geqslant \sum_{i=1}^h \alpha_i$, by \eqref{eq:prox}, and
$\alpha_i\geqslant0$ for each $i$.
Thus $\phi$ maps $\cL$ to a linear system $\cL'$ of degree
\[
\delta d-\alpha m -\sum_{i=1}^h \alpha_i m_i
\geqslant \delta d-\alpha m -(d-m)\sum_{i=1}^h \alpha_i
\geqslant \delta d-\alpha m -(d-m)\alpha=(\delta-\alpha)d\geqslant d.
\]
If equality holds, then $\alpha=\delta-1$,
$\alpha_i\leqslant1$ for each $i$,
$d=m+m_i$ whenever $\alpha_i=1$,
and there are exactly $\delta-1$ indexes $i$ such that $\alpha_i=1$.
This proves the assertion in case $(ii)$.
\end{proof}

If $\cL$ is a linear system of plane curves with no multiple
fixed components, then there is a
birational morphism $f\colon S\to \Pl$ such that the proper transform of the general curve $C$
of $\cL$ is smooth. Let $\cL'$ be the linear system on $S$ such that $(S,\cL')\sim
(\Pl,\cL)$ via $f^{-1}$.
We say that $\cL$ is \emph{complete} if so is $\cL'$. 
We will denote by $\ad_m(\cL)$ 
the linear system $f_*(\vert C+mK_S\vert)$ and we call it the
\emph{$m$-adjoint linear system} of $\cL$. This system 
is independent on $f$.
If $\cL=\{C\}$ has dimension zero, we write $\ad_m(C)$.
In case $m=1$, we write $\ad(\cL)$ and call it the
\emph{adjoint linear system} of $\cL$.
Note that $\deg(\ad_m(\cL))=\deg (\cL)-3m$.

\begin{remark}\label{rem:Cremonaminimal}
Whenever $\ad_m(\cL)$ and $\cL$ are not empty and 
$\ad_m(\cL)$ is Cremona
minimal, then also $\cL$ is Cremona minimal. 
\end{remark}

Using adjoints, one may easily prove a very useful, classical result due to Jung in \cite{Jung1}
(cf.\ \cite[p.\ 402--403]{Coolidge}):

\begin{theorem}[Jung]\label{thm:jung}
A linear system  $\cL=\cL(d;m_1,m_2,m_3,\ldots)$, without multiple fixed components,
with  $m_1\geqslant m_2\geqslant m_3 \geqslant \dots$ and $d\geqslant m_1+m_2+m_3$, is Cremona minimal.
\end{theorem}

\begin{proof}
By Lemma \ref{lem:criterion}, we may assume $m_3>0$.
Then $\ad_{m_3}(\cL)=\cL(d-3m_3;m_1-m_3,m_2-m_3)$ is Cremona minimal,
again by Lemma \ref{lem:criterion},
and so is $\cL$ (cf.\ Remark \ref{rem:Cremonaminimal}).
\end{proof} 

Furthermore, Lemma \ref{lem:criterion} implies that:

\begin{corollary}\label{cor: jung}
Let $\cL$ be as in Theorem \ref{thm:jung}.
If $d>m_1+m_2+m_3$ and $\phi\colon\Pl\dasharrow \Pl$ is a Cremona transformation such that
$\phi_*(\cL)$ has degree $d$, then $\phi$ is linear.
\qed
\end{corollary}

A linear system $\cL$ satisfying the hypotheses of Theorem \ref{thm:jung} will be called of \emph {Noether type}. 

\section{Properties of $(-1)$-cycles}\label{subsec:-1cicli}

Let $S$ be a surface and let $Z>0$ be a divisor on $S$. We will say that $Z$ is a \emph{$(-1)$-cycle} if there is a surface $S'$ and a birational morphism $f\colon S\to S'$ such that $f(Z)$ is a point $p\in S'$, $Z$ is the \emph{scheme theoretical fibre} of $f$ over $p$,
i.e.\ $\cO_S(-Z)\simeq f^*({\mathcal I}_{p\vert S'})$, and $f: S'-Z\to S-\{p\}$ is an isomorphism.
In this case we will say that $f$ \emph{blows down} $Z$.
 
The structure of a $(-1)$-cycle has been classically studied by Barber and Zariski \cite{BarberZariski} and Franchetta \cite{A2}. 
We will not need it here.

If $Z$ is irreducible, then it is the exceptional divisor of a blow-up. Otherwise $f$ is the composition of blow-ups
\begin{equation}\label{eq:seq}
f\colon S=S_0\to S_1\to \cdots \to S_{n-1}\to S_n=S'.
\end{equation}
If $j>i$, denote by $f_{i,j}$ the morphism $S_i\to S_j$,  and denote by  $Z_{i,j}$ the corresponding $(-1)$-cycle on $S_i$. Then $Z$ is the total transform of  $Z_{i,n}$ via the map $f_{0,i}$, for all  $i=1,\ldots,n-1$.
We will say that each $Z_{0,j}$ is \emph{$f$-exceptional}.

The proof of the following lemma is trivial:

\begin{lemma}\label{lem:ints} In the above setting, if $0\leqslant i<h<j\leqslant n$, one has 
${(f_{i,h}})_*(Z_{i,j})=Z_{h,j}$ and $Z_{i,j}\cdot Z_{i,h}=0$. 
\end{lemma}

The following lemmata are classical (see \cite{BarberZariski, A2}).

\begin{lemma}\label{lem:pullback}  Let $Z$ be a $(-1)$-cycle on a surface $S$, and let $f\colon S\to S'$ be the birational morphism blowing down $Z$. Then
\begin{equation}\label{eq:kappa}
K_S\equiv f^*(K_{S'})+\sum_{i=1}^n Z_{0,i}.
\end{equation}
\end{lemma}
\begin{proof} Proceed by induction on the number $n$ of blow-ups appearing in the sequence \eqref{eq:seq}.  \end{proof}

\begin{lemma}\label{lem:ozf} If $Z$ is a $(-1)$-cycle on a surface $S$, then 
\begin{equation}\label{eq:exc}
K_S\cdot Z=Z^2=-1 \quad\text{thus} \quad p_a(Z)=0.
\end{equation}
\end{lemma}
\begin{proof} Since $Z$ is the total transform of the 
exceptional curve $Z_{n-1,n}$ one has $Z^2=-1$. 
By intersecting both sides of \eqref{eq:kappa} with $Z$ and taking into account Lemma \ref{lem:ints},
one finds $K_S\cdot Z=-1$.
\end{proof}

By Zariski's Main Theorem, a $(-1)$-cycle is topologically connected. More precisely (see \cite{A5}):

\begin{lemma}  A  $(-1)$-cycle is numerically connected.
\end{lemma}

\begin{proof} Let $Z$ be a $(-1)$-cycle  and let $Z=A+B$ with $A,B>0$. Suppose that $A\cdot B<0$.
By Grauert's criterion (see \cite{badescu}, Corollary 2.7), the intersection matrix of $Z$ is negative definite. Thus we have $-1=Z^2=A^2+B^2+2A\cdot B\leqslant -4$, a contradiction.
\end{proof}

Franchetta proved in \cite{A5} that any numerically connected curve $Z$ on a surface $S$ with Kodaira dimension $\kappa(S)\geqslant 0$ verifying \eqref{eq:exc} is a $(-1)$-cycle. As noted by Nagata in \cite{Nagata2}, p.\ 282, the assertion is no longer true
if $\kappa(S)=-\infty$.  Nagata however  misunderstands Franchetta's statement. 

In what follows we will need the following technical lemmata.

\begin{lemma}\label{prop:zeroint} Let $Z, Z'$ be distinct $(-1)$-cycles on the surface $S$ and assume that the intersection matrix of the components of $Z+Z'$ is negative definite. Then:
\begin{enumerate}[(i)]
\item $Z\cdot Z'\geqslant 0$;
\item if $Z\cdot Z'=0$ and if $f\colon S\to S'$ is the birational morphism blowing down $Z$, then either $ f_*(Z')=0$, which  happens if and only if $Z> Z'$, or $f_*(Z')$ is a $(-1)$-cycle on $S'$.
\end{enumerate}
\end{lemma}

\begin{proof}  Assertion $(i)$ is clear if $Z, Z'$ have no common component. Let us proceed by induction on the number of common components of $Z$ and $Z'$. 

Consider the map $f\colon S\to S'$ blowing down $Z$ and suppose it is a sequence as in 
\eqref{eq:seq}. By Lemma \ref{lem:ints}, we have  $Z\cdot Z_{0,1}=0$. 

Assume first that  $Z_{0,1}\leqslant Z'$.  If $Z_{0,1}= Z'$, then $Z\cdot Z'=0$. Otherwise 
$Z'_1:={(f_{0,1})}_*(Z')$ is a $(-1)$-cycle
$Z'_1$ on $S_1$ and $Z'=f^*_{0,1}(Z'_1)$. By induction, we have
$$Z\cdot Z'=f^* _{0,1}(Z_{1,n})\cdot f^*_{0,1}(Z'_1)=Z_{1,n}\cdot Z'_1\geqslant 0.$$

Assume now $Z_{0,1}$ is not contained in $Z'$. One has $Z_{0,1}\cdot Z'=0$, otherwise
$(Z'+Z_{0,1})^2\geqslant 0$, against the negativity assumption on the  intersection matrix of the components of $Z+Z'$.  Then $Z_{0,1}$ and $Z'$ have no points in common and therefore 
${(f_{0,1})}_*(Z')$ is a $(-1)$-cycle on $S_1$, and we can repeat the argument. After a finite number of steps, we accomplish the proof of $(i)$.
The proof of $(ii)$ is analogous and may be left to the reader. \end{proof}

\begin{lemma}\label{lem:move} Let $Z, Z'$ be distinct $(-1)$-cycles on the surface $S$.
If $Z\cdot Z'>0$, then:
\begin{enumerate}[(i)]
\item if $q(S)=0$,  there is some divisor $D\leqslant Z+Z'$ which moves in a linear system of positive dimension on $S$;
\item if $q(S)>0$,  then the Albanese map of $S$ maps $S$ to a curve and $Z+Z'$ is a rational multiple of a fibre. 
\end{enumerate}
\end{lemma}

\begin{proof} If $q=0$, by the Riemann-Roch Theorem we have
$h^0(S,\cO_S(Z+Z'))\geqslant 1+Z\cdot Z'\geqslant 2$, and the assertion follows.

If $q>0$,  then the Albanese map contracts $Z+Z'$ to a point. Hence the intersection matrix of the components of $Z+Z'$ is negative semi-definite (see \cite{badescu}, Corollary 2.6 and 2.7). Since $(Z+Z')^2=2(Z\cdot Z'-1)$, we have  $Z\cdot Z'=1$ and $(Z+Z')^2=0$ and the assertion follows from Zariski's Lemma (see \cite{badescu}, Corollary 2.6). \end{proof}


\section{Basic properties of adjoint linear systems}\label{sec:adjprop}
 
\subsection{Nefness property of $m$-adjoint linear systems}
The following results are essentially known to the experts.
Since we have not found a proper reference, we quickly present them here. 

\begin{proposition}\label{prop:zariski}
Let $C$ be a divisor on a surface $S$ such that either
\begin{enumerate}[(i)] 
\item $C$ is nef; or
\item $C$ is effective, irreducible and not a $(-k)$-curve, with $1\leqslant k\leqslant 3$. 
\end{enumerate}
Let $m> 0$ be an integer. Suppose that $\vert C+mK_S\vert\ne\emptyset$ and that $m\geqslant 2$ if we are in case (ii) and $C$ is rational. 

Then there is a surface $S'$ and a birational morphism 
$f\colon S\to S'$ such that $\vert C'+mK_{S'}\vert$ is nef, where $C'=f_*(C)$, and 
\begin{equation}\label{eq:birinv}
h^0(S,\cO_S(C+mK_S))=h^0(S', \cO_{S'}(C'+mK_{S'})).
\end{equation}
\end{proposition}

\begin{proof} If $C+mK_S$ is nef, the assertion is clear. 
Suppose there are irreducible curves $\Theta>0$ such that 
$\Theta\cdot (C+mK_S)<0$. Since $C+mK_S\geqslant 0$, there are only finitely many such curves
and each of them has $\Theta^2<0$. We claim 
they are  $(-1)$-curves. Indeed, if $(i)$ holds, then
$0>\Theta\cdot (C+mK_S)=\Theta\cdot C+m(\Theta\cdot K_S)\geqslant m(\Theta\cdot K_S)$, hence
$\Theta\cdot K_S<0$ and therefore  
$\Theta^2=-1$ and $p_a(\Theta)=0$. If $(ii)$ holds, one has
$C\cdot \Theta\geqslant 0$ and therefore
the same conclusion holds. Indeed, if $C\cdot \Theta< 0$
then $C=\Theta$ and 
$0>\Theta\cdot (C+mK_S)=2m(p_a(C)-1)-C^2(m-1)$. This implies $p_a(C)=0$ hence $m\geqslant 2$ and
$C^2>-2-\frac{2}{m-1}$, thus $C^2\geqslant -3$, a contradiction.

Let  $\Theta$ be a $(-1)$-curve such that $\Theta\cdot (C+mK_S)<0$. Set $\Theta\cdot C=i\geqslant 0$ and note that $i<m K_S\cdot \Theta=m$.
We have a birational morphism 
$\pi\colon S\to S_1$ blowing down $\Theta$ to a smooth point $p$. 
Set $C_1=\pi_*(C)$. Then $C=\pi^*(C_1)-i\Theta$ and $K_S\equiv \pi^*(K_{S_1})+\Theta$,
hence
$$
C+mK_S\equiv \pi^* (C_1+mK_{S_1})+ (m-i)\Theta.
$$
and
$$h^0(S_1, \cO_{S_1}(C_1+mK_{S_1})=h^0(S,\cO_S(C+mK_S)).$$
So $C_1+mK_{S_1}$ is effective.
By repeating this  argument, we see that there is a sequence of blow-downs as in \eqref{eq:seq}, such that  $C'+mK_{S'}$ is nef, where $C'=f_*(C)$ and \eqref{eq:birinv} holds.
\end{proof}

Consider all $f$ exceptional $(-1)$-cycles. We may index 
them with  two indices as  $\Theta_{i,j}$, where  $i=C\cdot \Theta_{i,j}$  and $j=1,\ldots, h_i$. Then 
\begin{equation}\label{eq:C+mK0}
C=f^*(C')-\sum _i \sum_{j=1}^{h_i}i\Theta_{i,j} \quad\text{and} \quad K_S=K_{S'}+\sum_i \sum_{j=1}^{h_i}\Theta_{i,j}. 
\end{equation}

The following proposition specifies the Zariski decomposition of the adjoint systems we are interested in. 

\begin{proposition} \label{prop:zariski2} Same hypotheses of Proposition \ref{prop:zariski}. 
Then we have a unique expression
\begin{equation}\label{eq:C+mK}
C+mK_S\equiv P+\sum_{i=0}^{m-1}\sum_{j=1}^{h_i}(m-i)\Theta_{i,j},
\end{equation}
where $P$ is nef and the $\Theta_{i,j}$'s are $(-1)$-cycles such that
\begin{enumerate}[(i)]

\item $P\cdot \Theta_{i,j} = 0$, for each $(i,j)$;

\item $\Theta_{i,j} \cdot \Theta_{h,k}=0$, for $(h,k)\ne(i,j)$;

\item the intersection matrix of the irreducible components of 
$$N=\sum_{i=0}^{m-1}\sum_{j=1}^{h_i}(m-i)\Theta_{i,j}$$ 
is negative definite;

\item $C \cdot \Theta_{i,j} = i$, for each $(i,j)$;

\item $h^0(S,\cO_S(C+mK_S))=h^0(S, \cO_S(P))$.
\end{enumerate}

\end{proposition}

\begin{proof} It follows from Proposition \ref{prop:zariski} and from \eqref{eq:C+mK0}. Note that 
$P=f^*(C'+mK_{S'})$ is nef. Then $(i)$ is clear, $(ii)$ follows
by Lemma \ref{lem:ints}, $(iii)$ follows by \cite{badescu}, Corollary 2.7, $(iv)$ holds by definition and $(v)$ follows by Proposition \ref{prop:zariski}.
Since $i-m=\Theta_{i,j}\cdot (C+mK_S)<0$, we have $i<m$.

As announced, \eqref{eq:C+mK} is nothing else than the Zariski decomposition of $C+mK_S$ (see \cite{badescu}, Chapt. 14). As such it is unique.
\end{proof}

The divisors $N, P$ appearing in $(iii)$ of Proposition \ref{prop:zariski2} are the \emph{negative part} and the \emph{nef part}, respectively of $C+mK_S$.

\begin{remark} \label{rem:stable} The decomposition \eqref{eq:C+mK} is \emph{stable} under birational morphisms, in the following sense. Let $X$ be a surface, let $f\colon X\to S$ be a birational morphism, and let $\Gamma$ be the proper transform on $X$ of the curve $C$ on $S$. 

Let $\Theta_{i,j}$ be the $f$-exceptional $(-1)$-cycles  on $X$, indexed in such a way that
 $\Theta_{i,j}\cdot \Gamma=i$, $j=1,\ldots, k_i$.  Then
$$\Gamma=f^*(C)-\sum_i\sum_{j=1}^{k_i}i\Theta_{i,j}\quad\text{and}
\quad K_X\equiv f^*(K_S)+\sum_i\sum_{j=1}^{k_i}\Theta_{i,j}$$
so that
$$\Gamma+mK_X\equiv f^* (C+mK_S)+\sum_i\sum_{j=1}^{k_i}(m-i)\Theta_{i,j}.$$
From this we deduce that
$$h^0(S,\cO_S(C+mK_S))\geqslant h^0(X,\cO_{X}(\Gamma+mK_{X}))$$
and equality holds if each $i$ is at most $m$.
Assume this is true. This is  the case if $C$ has points of multiplicity at most $m$.
Then, substituting for $C+mK_S$ the expression given by  \eqref{eq:C+mK} we find the analogous decomposition for $\Gamma+mK_X$. 

More generally, if $n\geqslant 1$ is an integer, we have
$$n\Gamma+mK_X\equiv f^* (nC+mK_S)+\sum_i\sum_{j=1}^{k_i}(m-ni)\Theta_{i,j}.$$
and 
$$h^0(S,\cO_S(nC+mK_S))\geqslant h^0(X,\cO_{X}(n\Gamma+mK_{X}))$$
with equality if $m\geqslant ni$ for all $i$.
\end{remark}

\subsection{Adjoint systems and birational transformations}
\label{susec:birat}
 
Suppose that $(S,C)\sim (S',C')$ via the birational map $\phi\colon S\rto S'$.
Then  there is a commutative diagram
 \[
\xymatrix{ & X \ar[dl]_{f} \ar[dr]^{f'} \\
S\ar@{-->}[rr]^{\phi} & & S'}
\]
where $X$ is a surface, $f,f'$ are birational morphisms and there is a curve $\Gamma$ on $X$ such that
$f_*(\Gamma)=C$, $f_*'(\Gamma)=C'$.

As we saw
\begin{equation}\label{eq:double}
\Gamma+mK_X\equiv f^* (C+mK_S)+\sum_i\sum_{j=1}^{h_i}(m-i)\Theta_{i,j}\equiv f'^*(C'+mK_{S'})+
\sum_u\sum_{v=1}^{h'_u}(m-u)\Theta'_{u,v}
\end{equation}
and therefore
\begin{equation}\label{eq:phi!}
C'+mK_{S'}\equiv  \phi_! (C+mK_S)
+ \sum_i\sum_{j=1}^{h_i}(m-i)f'_*(\Theta_{i,j}).
\end{equation}

\begin{proposition}\label{prop:bir}
In the above setting, assume $C,C'$ verify the hypotheses of Proposition \ref{prop:zariski} and have points of multiplicity at most $m$. Then 
$$h^0(S,\cO_S(C+mK_S))=h^0(S',\cO_{S'}(C'+mK_{S'})).$$
If, in addition, $C,C'$ and have points of multiplicity at most $m-1$, then
  $$C'+mK_{S'}\equiv  \phi_! (C+mK_S)+ \sum_{i=1}^{m-1}\sum_{\ell=1}^{k_i}(m-i) \theta_{i,\ell}$$ 
  where $\theta_{i,j}$ are $(-1)$-cycles on $S'$ such that $C'\cdot \theta_{i,\ell}=i$, which are the images of 
$f$-exceptional $(-1)$-cycles. 
 \end{proposition}

\begin{proof} The first assertion follows by Remark \ref{rem:stable}.

As for the second, from \eqref{eq:double} we see that all the $f'$-exceptional $(-1)$-cycles $\Theta'_{u,v}$ sit in the negative part of $\Gamma+mK_X$. Then $\Theta_{i,j}\cdot \Theta'_{u,v}=0$. By part $(ii)$ of Lemma \ref{prop:zeroint}, we see that $\theta_{i,j}:=f'_*(\Theta_{i,j})$ are either zero or $(-1)$-cycles.
 Moreover $\theta_{i,j}\cdot  \phi_! (C+mK_S)=0$, which proves that $C'\cdot \theta_{i,\ell}=i$.  \end{proof}
 
In the hypotheses of Proposition  \ref{prop:bir}, we set
$$a_m(C):=h^0(S,\cO_S(C+mK_S))$$
which is a birational invariant of the pair $(S,C)$.

\begin{remark}\label{rem:multad}  Let $n \geqslant 1$ be an integer.
Assume $C,C'$ verify the hypotheses of Proposition \ref{prop:zariski}, from which we keep the notation, and have points of multiplicity at most $[\frac mn]$.   In particular $m\geqslant n$.
Assume $h^0(S,\cO_S(nC+mK_S))>0$. Then 
$$nC'+mK_{S'}\equiv  \phi_! (nC+mK_S)+ \sum_{i=1}^{[\frac mn]}\sum_{\ell=1}^{k_i}(m-ni) \theta_{i,\ell}$$ 
 and
 $$h^0(S,\cO_S(nC+mK_S))=h^0(S',\cO_{S'}(nC'+mK_{S'}))$$
so that
$$a_m(nC):=h^0(S,\cO_S(nC+mK_S))$$ 
is a also birational invariant of the pair $(S,C)$ (see \cite{Iitaka2}).

More specifically, the projective isomorphism class of the image
of $S$ via $\phi_{|nC+mK_S|}$ is a birational invariant.
\end{remark}

\begin{remark}\label{rem:sistad} All the above result hold for pairs $(S,\cL)$ with $\cL$ a linear system as well.
\end{remark}

Let $(S,C)$ be a pair  as usual, with $S$ smooth and $C$ with points of multiplicity at most $m$. The expert  reader  in higher dimensional birational geometry, will commonly express this by saying that the pair $(S,\frac C m)$ has \emph {canonical singularities},
and actually \emph {terminal singularities} if $C$ has points of multiplicity at most $m-1$ (see \cite{Mats}). Though most fashionable today, we will not use this terminology here, since we feel that in the surface case the classical one is equally appropriate. 

Then one defines  the \emph {Kodaira dimension} $\kappa(S,\frac C m)$ of the pair $(S,\frac C m)$ to be $-\infty$ if $a_{mn}(nC)=0$ for all $n>0$; otherwise one defines $\kappa(S,\frac C m)$  to be the dimension of the image of $\phi_{|n(C+mK_S)|}$, for $n>>0$.  As we saw, this is a birational invariant.  

If $\cL$ is a complete planar linear system with no multiple fixed components, and $f: S\to \PP^2$ is a birational morphism such that the general member $C$ of proper transform $\cL'$ of $\cL$ on $S$ is smooth,  we can consider  $\kappa(S,C)$. This is an invariant of $\cL$ by Cremona transformations, which is called the \emph{Kodaira dimension} of $\cL$. We will denote it by $\kappa(\cL)$.

\section{Curves on rational  ruled surfaces}\label{sec:Iitaka}

\subsection{Minimal rational surfaces and elementary transformations}\label{ssec:minimalrat}

The minimal rational surfaces are $\Pl$ and
$\F_n=\PP(\cO_{\PP^1}\oplus \cO_{\PP^1}(n))$, $n\geqslant0$ and $n\ne 1$.
The surface $\F_n$ is a $\PP^1$-bundle on $\PP^1$. We denote
by $f_n: \F_n\to \PP^1$ the structure morphism,
by $F_n$ its general fibre,
by $F_p$ the fibre through a point $p\in \F_n$.
There is a section $E_n$ with $E_n^2=-n$, which is uniquely determined
if $n>0$. If $n=0$ then $\vert E_0\vert$ is a pencil.
We will drop the index $n$ if there is no danger of confusion.
Recall that $\Pic(\F_n)$ is generated by the classes of $E$ and $F$.

On $\F_n$, we denote by
\begin{equation}\label{eq:sist}
\cL_n(k,h;[\mu_1^{\eps_1},\mu_2^{\eps_2},\ldots,\mu_s^{\eps_s}],
m_1^{e_1},\ldots,m_r^{e_r})
\end{equation}
the linear system of curves in $|kE+hF|$ having
$\eps_i$ (proper or infinitely near) points of multiplicity at least $\mu_i$, $i=1,\ldots,s$, on $E$ (or on its proper transform)
and further $e_j$ points of multiplicity at least $m_j$, $i=0,\ldots,r$.

If $\cL(d;m_0,\ldots,m_r)$ with $m_0\geqslant \ldots \geqslant m_r$, we may
blow up the base points $p_0$ of multiplicity $m_0$ and take the
proper transform of the system to $\F_1$, thus getting the system 
$\cL_1(d-m_0,d; m_1, \ldots,m_r)$. In this way we will look at planar 
linear systems as linear systems on $\F_1$. 

The surfaces $\F_n$ can be realised in a projective space as surfaces of 
minimal degree (cf.\ \cite{delPezzo}, \cite{Conforto}, \cite{EisenbudHarris}).
Consider on $\F_n$ the base point free linear system $\vert E+hF\vert$, with $h\geqslant n$.
If $h=n=0$, then $\vert E\vert $ is a pencil.
Otherwise it determines
a morphism $\phi=\phi_{\vert E+hF\vert}\colon \F_n\to \PP^r$,
with $r=2h-n+1$, whose image is a possibly singular surface $\Sigma$ of degree
$r-1=2h-n$.
In fact $\phi\colon \F_n\to \Sigma$ is an isomorphism as soon as $h>n$. In
case $h=n>0$, $\Sigma$ is the cone over a rational normal curve of degree 
$r-1$ and $\phi$ is an isomorphism off the negative curve $E$, which is
contracted to the vertex of the cone. 
The surface $\Sigma$ is denoted as $S(h-n,h)$, since it is the scroll described
by the lines joining corresponding points on two rational normal curves
of degree $h-n$ and $h$ in independent spaces. This applies also in
the degenerate case $h=n$. 


An \emph{elementary transformation} $\elm_p$ centered at a point $p\in\F_n$ is
the composition of the blowing up of $\F_n$ at $p$ and
the blowing down of the proper transform of the fibre through $p$,
which is a $(-1)$-curve.
If $p\not\in E$, then $\elm_p\colon \F_n\rto \F_{n-1}$,
otherwise $\elm_p\colon \F_n\rto \F_{n+1}$.

The composition $\elm_{p_r} \circ \cdots \circ \elm_{p_2} \circ \elm_{p_1}$
of $r$ elementary transformations centered at points $p_1,p_2,\ldots,p_r$
will be denoted by $\elm_{p_1,\ldots,p_r}$. 
Only the  point $p_1$ belongs to $\F_n$ whereas $p_i\in \elm_{p_1,\ldots,p_{i-1}}(\F_n)$,
for all $i=2,\ldots, r$.


A birational map $\phi: \F_n\rto \F_{n'}$ sending the pencil $\vert F_n\vert$ to
the pencil $\vert F_{n'}\vert $, i.e.\ making the following diagram commutative
\begin{equation}\label{eq:commute}
\raisebox{3.5ex}{\xymatrix{%
\F_{n}  \ar[d]_{f_n}  \ar@{-->}[r] ^{\phi} & \F_{n'} \ar[d]^{f_{n'}} \\
\PP^1 \ar[r]^{\id}  & \PP^1}}
\end{equation}
is called a \emph{fibred birational map}. For example, de Jonqui\`eres
transformations can be regarded as fibred birational maps $\F_1 \rto \F_1$.

Any fibred birational map is a composition of elementary transformations. 
This is an aspect of Sarkisov's  program in dimension 2 (see
\cite{Mats}).

\subsection{Iitaka's $\sharp$-models (sharp models)}

Consider a pair $(\F_n,\cL)$, with $\cL$ as in \eqref{eq:sist}
and call $m_1$ the maximum multiplicity of singular points of $\cL$.
We will assume in the sequel  that $h\geqslant kn$, so that the curve
$E_n$ does not split off from the system  $\cL$. 
Following Iitaka \cite{Iitaka1, Iitaka2}, we say that the pair
$(\F_n,\cL)$ is \emph{$\sharp$-minimal}, or a \emph{$\sharp$-model}
[resp.\ \emph{$\sharp\sharp$-minimal}, or a \emph{$\sharp\sharp$-model}] when:
\begin{itemize}

\item $k\geqslant 2m_1$ [resp.\ $k>2m_1$] if $n\geqslant 2$;

\item $k\geqslant 2m_1$ [resp.\ $k>2m_1$] and $h-k\geqslant m_1$ if $n=1$;

\item $h,k\geqslant 2m_1$ [resp.\ $h,k>2m_1$] if $n=0$.
\end{itemize}

This is essentially the same as saying that  the pair $(\F_n,\frac {2C}k)$ has canonical [resp. terminal] 
singularities. 

If a pair $(\F_n,\cL)$ is not $\sharp$-minimal,
one may get a $\sharp$-model by performing elementary transformations
based at points of multiplicity $\mu$ with $2\mu>k$.

In his papers \cite{Iitaka1, Iitaka2},
Iitaka described many interesting properties of $\sharp$-models 
and $\sharp\sharp$-models.
In particular the following result will be useful for us (see \cite{Iitaka2},
Theorem 4, its proof and its corollary, on p.\ 316).

\begin{theorem}[Iitaka] \label{thm:iitaka}  One has:
\begin{enumerate}[(i)]

\item if $(\F_n,\cL)\sim (\F_{n'},\cL')$, 
$(\F_n,\cL)$ is a $\sharp\sharp$-model and
$(\F_{n'},\cL')$ is a $\sharp$-model, then 
$(\F_n,\cL)$ and $(\F_{n'},\cL')$ are isomorphic;

\item  if $(\F_n,\cL)$, with $\cL\subset\cL_n(k,h)$, and $(\F_{n'},\cL')$
are both $\sharp$-models and are birationally
equivalent via a map $\phi$ fitting in a diagram
\eqref{eq:commute},
then either $\phi$ is an isomorphism, or
$k=2m$ is even and $(\F_{n'},\cL')$ is obtained from $(\F_n,\cL)$
with a sequence of elementary transformations based 
at points of multiplicity $m$. 
If $\cL=\cL_n(k,h;
m_1,\ldots,m_r)$, then $m_1,\ldots,m_r$ are
birational invariants of $\sharp$-models birationally
equivalent to $(\F_n,\cL)$.
\end{enumerate} 
\end{theorem}

\begin{proof}  Part (i) is the uniquess of terminal models. In any event, set $m=m_1$.
One has $\ad_m(\cL)=
\cL_n(k-2m,h-m(n+2))$. This system is very ample on
$\F_n$, except for $h=kn$, in which case it is very ample
on $S(0,n)$.  By Remark \ref{rem:multad}, part $(i)$ immediately follows.
Part $(ii)$ follows by a similar argument, left to the reader,
 applied to $\ad_{m-1}(\cL)$.
\end{proof}

\subsection{Definition of $\flat$-models (flat models) and $\natural$-models (natural models)} \label{ssec:model}
Given a pair $(\F_{n'},\cL')$ as above, there is a minimal $n\geqslant 0$ such that
$(\F_n,\cL)$ is a $\sharp$--model birational to $(\F_{n'},\cL')$. We call it a \emph{$\flat$--model}
and $n$ the \emph{$\flat$--index} of  $(\F_{n'},\cL')$. 

\begin{proposition}\label{pro:flat} Let $(\F_n,\cL)$ be a $\flat$--model with $\cL$ as in  \eqref{eq:sist}. If $n>0$, then:
\begin{enumerate}[(i)]
\item  all base points of $\cL$ of multiplicity $m\geqslant \frac k2$ lie on $E_n$;
\item if  $(\F_n,\cL')$ is a $\flat$--model  birational to $(\F_n,\cL)$
via a map $\phi$ fitting in a diagram
\eqref{eq:commute}, and either $k$ is odd or $\cL$ does not have  infinitely near base points  of multiplicity $m$ on $E_n$, then $\phi$ is an isomorphism.
\end{enumerate}
\end{proposition}

\begin{proof} By Theorem \ref{thm:iitaka}, the assertion is clear if $k$ is odd, since in this
case $(\F_n,\cL)$ is a $\sharp\sharp$--model. So assume $k$ is even. 

Part (i) follows by
the very definition of a $\flat$--model, since an elementary transformation performed at a base point of multiplicity $\frac k2$ off $E_n$, would drop $n$ and keep sharpness. 

As for part (ii), again by  Theorem \ref{thm:iitaka}, the map $phi$ is a composition of elementary transformations based at points of multiplicity $m=\frac k2$. Each of them preserves sharopness, so this series of maps never involves $\F_m$, with $m<n$. On the other hand, by part (i), each elementary transformation creating a point of multiplicity $m$ off the negative curve, has to be compensated by another elementary transformation at a point of multiplicity $m$ sending this back on the negative curve. 
The hypothesis about the points of multiplicity $m$ on $E_n$ shows that each of these transformations is the inverse of a previous one.  \end{proof}

\begin{remark}\label{rem:noun} The 
assumption $n>0$ and about the points of multiplicity $m$ on $E_n$ in Proposition \ref{pro:flat} is essential,  as one sees by considering linear systems of the form $\cL_0(2m,h;m^e)$, with $h\geqslant n$ and $e\geqslant 2$ and  $\cL_n(2m,h;m^2)$, with $h\geqslant mn$ and the second point of multiplicvity $m$ is infinitely near to the first proper point on $E_n$, in a direction which is not tangent to the one of $E_n$.
\end{remark}

Let $(\F_n,\cL)$ be a pair with $\cL$ as in \eqref{eq:sist}.
Again we will assume  that $h\geqslant kn$. 
We say that the pair
$(\F_n,\cL)$ is \emph{$\natural$-minimal}, or a \emph{$\natural$-model},
when:
\begin{itemize}

\item $k=2m+\epsilon$, $m\geqslant 1$, $\epsilon=0,1$;

\item there is no base point of $\cL$ of multiplicity $\mu\geqslant 2m+\epsilon-1$;

\item  there is no base point of multiplicity $\mu\geqslant 2$ along $E_n$.
\end{itemize}

A $\natural$-model certainly exists and it is obtained, for example, from a $\sharp$-model
by a sequence of elementary trasformations based at the singular base points
of $\cL$ on $E$. 

\begin{proposition} \label{prop:percent}  If $(\F_n,\cL)$ and $(\F_{n'},\cL')$
are both $\natural$-models with positive $\flat$--index and are birationally
equivalent via a map $\phi$ fitting in a diagram
\eqref{eq:commute}, then $\phi$ is an isomorphism.
\end{proposition}

\begin{proof} Let $(\F_m,\tilde \cL)$ and $(\F_{m'},\tilde \cL')$ be $\flat$--models birational to
$(\F_n,\cL)$ and $(\F_{n'},\cL')$ via fibred birational maps. Then, 
by Proposition \ref{pro:flat} and the positivity assumption on the $\sharp$--index, they are isomorphic. Moreover $(\F_n,\cL)$ and $(\F_{n'},\cL')$ are obtained by this unique model $(\F_m,\tilde \cL)$ by performing elementary transformations based at the singular base points of $\tilde \cL$ along $E_m$.  The assertion follows. 
\end{proof}

\section{Linear systems of rational curves}\label{subsec:moving}

In this section we recall some basic results about complete linear systems
of rational curves, in particular we are  interested
in a plane birational model with minimal degree.
This goes back to the classical
Italian school (cf.\ \cite{Conforto} for results and references).
For a recent reference, see \cite{Iitaka2}, p.\ 192.

\begin{proposition}\label{D+K=0}
Let $S$ be a smooth, irreducible, projective surface with $q(S)=0$
and let $D$ be a nonzero nef divisor with $h^0(S,\cO_S(D+K_S))=0$.
Then $S$ is rational and either

\begin{enumerate}[(i)]

\item $D^2=0$ and $|D|$ is composed with an irreducible base point free
pencil of rational curves; or

\item $D^2>0$ and $|D|$ is an irreducible base point free linear system
of rational curves of dimension $D^2+1$.
\end{enumerate}
\end{proposition}

\begin{proof}
Assume first $D^2=0$.
By the Riemann-Roch Theorem, one has
\[
0=h^0(S,\cO_S(D+K_S))\geqslant 1+D\cdot K_S/2,
\]
so $D\cdot K_S\leqslant -2$.
Again by the Riemann-Roch Theorem, one has $h^0(S,\cO_S(D))\geqslant 2$.
Write now $|D|=F+|M|$, where $M$ is the movable part.

By the nefness property of $D$, we have $F\cdot D=M\cdot D=0$,
hence $F^2=M^2=-F\cdot M$.
If $F\cdot M>0$, then $M^2<0$, a contradiction.
Since $M$ is nef, we have $F^2=M^2=F\cdot M=0$.
Then $M$ is composed with a base point free irreducible pencil $|L|$ with $p_a(L)=0$.
By Noether's Theorem, $S$ is a rational surface.

We claim that $F=0$.
Otherwise, since $L\cdot F=0$, by Zariski's Lemma
we have two relatively prime positive integers $p,q$
such that $qF\equiv pL$. Since $K\cdot L=-2$ and $K\cdot F$ is even, we have $q=1$,
hence $F\equiv pL$, a contradiction.

Assume now $d=D^2>0$.
Since $D$ is nef and big, it is also numerically connected.
Since $h^0(S,\cO_S(D+K_S))=0$ and $q=0$, one has $p_a(D)=0$.
By the Riemann-Roch Theorem, one has $h^0(S,\cO_S(D))\geqslant D^2+2$.
Let $p_1,p_2,\ldots,p_d$  be general points on $S$
and consider the sublinear system $\cD\subset |D|$ consisting
of all curves passing through $p_1,\ldots,p_d$.
Blow up $p_1,\ldots,p_d$, consider the proper transform $\cD'$
of $\cD$ on the blow-up  $S'$ and take a general curve $C$ in $\cD'$.
Then $C^2=0$ and one sees that $C$ is still nef and numerically connected.

In the short exact sequence
\[
0 \to H^0(S',\cO_{S'})\simeq \C \to H^0(S',\cO_{S'}(C)) \to H^0(C,\cO_C(C))
\to 0=H^1(S',\cO_{S'}),
\]
one has $h^0(S',\cO_{S'}(C))\geqslant 2$, hence $h^0(C,\cO_C(C))\geqslant 1$.
By Corollary (A.2) in \cite{CilibertoFranciaMendesLopes},
one has $\cO_C(C)\simeq \cO_C$ and therefore $h^0(C,\cO_C(C))=1$
and $|C|=\cD'$ is a base point free pencil whose general member
is irreducible.
This implies that $h^0(S,\cO_S(D))=D^2+2$.
\end{proof}

Next we recall the description of Cremona minimal birational models
of pairs $(S,|D|)$, with $S$ and $D$ as in Proposition \ref{D+K=0}.

\medskip
\noindent
\textbf{Case $\mathbf{D^2=0}$.}
As we saw, $|D|$ is composed with a base point free pencil $|L|$ of
rational curves.

By blowing down all $(-1)$-cycles $Z$ such that $Z\cdot L=0$,
we have a birational morphism $f\colon S\to \F_n$, for some $n$,
which maps $\vert L\vert $ to the ruling $\vert F\vert$ of $\F_n$ (one of the rulings if $n=0$).
By making elementary transformations at general points,
we see that 
\[
(S,|L|)\sim (\F_n,|F|)\sim (\F_1,|F|).
\]
By contracting the $(-1)$-section, this in turn is birationally equivalent
to $(\Pl,\cL(1;1))$.

\medskip
\noindent
\textbf{Case $\mathbf{D^2=d>0}$.}
The linear system $|D|$ determines a birational
morphism $\varphi\colon S \to \Sigma\subseteq \PP^{d+1}$,
where $\Sigma$ is a surface of degree $d$.
According to the del Pezzo's classification
of minimal degree projective surfaces
(cf.\ \cite{delPezzo}, \cite{Conforto}, \cite{EisenbudHarris}),
we have the following possibilities:
\begin{enumerate}[$\ \bullet\ $]

\item $\Sigma\simeq\Pl$ and $d=1$ or $4$;
accordingly, either $(S,|D|)\sim (\Pl,\cL(1))$ or $(S,|D|)\sim (\Pl,\cL(2))$;

\item $d=a+b$, with $0\leqslant a\leqslant b$, and $\Sigma=S(a,b)$, the minimal resolution
of singularities of $\Sigma$ is $\F_{b-a}$ fitting in a diagram
\[
\xymatrix@C+4ex{
S  \ar[r]^-{f} \ar[dr]_-{\varphi}  &  \F_{b-a} \ar[d]^-{\phi}\\
& \Sigma \subset \PP^{d+1} \hspace{-7ex}
}
\]
where $f$ is a birational morphism and $\phi$ is determined by 
the linear system $|E+bF|$,
hence $(S,|D|)\sim (\F_{b-a},\cL_{b-a}(1,b))$.

If $b>a$, make $b-a-1$ elementary transformations
based at general points of $\F_{b-a}$.
In this way
\[
(S,|D|)\sim (\F_{b-a},\cL_{b-a}(1,b))\sim (\F_1,\cL_1(1,b;[1^{b-a-1}])),
\]
where the $b-a-1$ simple base points are general on $E_1$.
By contracting $E_1$, one has 
\[
(S,|D|)\sim (\F_1,\cL_1(1,b;[1^{b-a-1}]))
\sim (\Pl,\cL(b;(b-1,[1^{b-a-1}])),
\]
where the $b-a-1$ simple base points are general
in the first order infinitesimal neighbourhood
of the point of multiplicity $b-1$.

If $b=a$, make one elementary transformation
based at a general point of $\F_0$ and then contract the $(-1)$-section.
In this way
\[
(S,|D|)\sim (\F_0,\cL(1,b))\sim (\F_1,\cL_1(1,b+1;1))
\sim (\Pl,\cL(b+1;b,1)).
\]
\end{enumerate}

Now we can state the following result:

\begin{theorem} \label{thm:rational}
Let $(S,D)$ be as in Proposition \ref{D+K=0}.
Then $(S,|D|)$ is birationally equivalent to one and only one of the following:
\begin{itemize}

\item [$(a)$] $(\F_1,\cL_1(0,1))$;

\item [$(b^j)$] $(\Pl,\cL(j))$, $j=1,2$;

\item [$(c_1^d)$] $(\F_1,\cL_1(1,d))$, $d\geqslant2$;

\item [$(c_0^d)$] $(\F_0,\cL_0(1,d-1))$, $d\geqslant2$;

\item [$(c_n^d)$] $(\F_n,\cL_n(1,d))$, $2\leqslant n\leqslant d$.
\end{itemize}
These, in turn, are birationally equivalent to the following
Cremona minimal pairs $(\Pl,\cL)$, where $\cL$ is one,
and only one, of the following:
\begin{enumerate}[(i)]

\item $\cL(1;1)$, of dimension 1, corresponding to case $(a)$;

\item $\cL(j)$, $j=1,2$, of dimension 2 and 5, resp.,
corresponding to case $(b^j)$;

\item $\cL(d;d-1)$, $d\geqslant 2$, of dimension $2d$,
corresponding to case $(c_1^{d})$;

\item $\cL(d;d-1,1)$, $d\geqslant 2$, of dimension $2d-1$,
corresponding to case $(c_0^{d})$;

\item $\cL(d;(d-1,[1^{n-1}]))$, $2\leqslant n\leqslant d$, of dimension $2d-n+1$, with the simple base points 
iinfinitely near of order one to the point of multiplicity $d-1$,
corresponding to case $(c_n^{d})$.
\end{enumerate}
\end{theorem}

\begin{proof}
Part  of the proof  is contained in the book \cite{Conforto}
by Conforto-Enriques and in Nagata's papers \cite{Nagata1, Nagata2},
with classical methods
(cf.\ the paper \cite{Dicks} by Dicks which uses Mori's theory).
Since we could not find a proper reference for the full statement,
we give a proof here.

The birational equivalence to one of the pairs in $(a)$, $(b^j)$, $(c_n^d)$,
has been proved above, cf.\ also \cite{Dicks}.
Pairs $(a)$, $(b^j)$, $(c_0^d)$, $(c_1^d)$, and
$(c_{n}^d)$, ${n\geqslant2}$, are respectively birationally equivalent
to pairs $(i)$, $(ii)$, $(iii)$, $(iv)$, and $(v)$.
Their Cremona minimality follows from Lemma \ref{lem:criterion}.

Finally, we prove that pairs $(\Pl,\cL)$, with $\cL$ as in $(i)$--$(v)$,
are birationally distinct. To see this, taking into account
the Cremona minimality, it suffices to remark
that if these linear systems  have the same dimension,
then they have different Cremona degrees.
\end{proof}

\begin{remark}
All linear systems in Theorem \ref{thm:rational},
except $\cL(1)$ and $\cL(1;1)$, are such that
$\Cdeg(\cL)> \Cdeg(C)$, where $C\in\cL$ is the general curve,
since clearly $\Cdeg(C)=1$.
\end{remark}


\section{Admissible plane models and Cremona transformations}\label{admissible}


Let $(\F_n,\cL')$ be a pair where $n\geqslant2$
and $\cL'= \cL_n(k,h;m_1,\ldots)$ not empty, where the indicated multiplicities are the
effective ones. 
By performing $n-1$ elementary trasformations at general points
$p_1,\ldots,p_{n-1}$ of $\F_n$,
and by  blowing down $\sigma\colon\F_1\to\Pl$ the $(-1)$-section $E_1$
(cf.\ Case $D^2>0$ before Theorem \ref{thm:rational}),
one gets a plane model $(\Pl,\cL)$ with
\[
\cL=\cL(h;(h-k,[k^{n-1}]),m_{1},\ldots).
\]
It may happen that some of the points of multiplicity $m_1,\ldots,$
are infinitely near to the point of multiplicity $h-k$.
This occurs if and only if some of the base points of 
$\cL'$ lie on $E_n$.

More generally, one can make the following process.
We consider a fibred birational map $\gamma\colon\F_n\rto\F_1$ such that
$\gamma(E_n)=E_1$.
In particular, $\gamma$ is a sequence of elementary trasformations.
By blowing down the $(-1)$-curve $E_1$ to a point $p$, one gets a planar linear system 
$\cL$ having multiplicity $\deg(\cL)-k$ at $p$.

In general, a pair $(\Pl,\cL')$ with 
\[
\cL=\cL(h;h-k,\nu_{1},\ldots).
\]
is said to be \emph{admissible}
if all base points $p_i$ of $\cL$ with multiplicity $\nu_i>k/2$, $i=1,\ldots,r$,
are infinitely near to the base point $p$ of multiplicity $h-k$ and
\[
h-3  \left[\frac{k}{2}\right] > \sum_{i=1}^r \nu_i-r\left[\frac{k}{2}\right].
\]
This implies that $p$ is the point of maximal multiplicity of $\cL$.

Note that $\cL$ is admissible if and only if all its multiples 
\[
\cL(ih;i(h-k),i\nu_{1},\ldots), \quad i\geqslant 1,
\]
are admissible.

If $k=2m$, the $m$-adjoint of  $ \cL$ is
\begin{equation}\label{eq:adm}
\ad_m(\cL)=\cL(d;(d,\{\mu_1,\ldots,\mu_r\}))
\end{equation}
where $\mu_i=\nu_i-m$, $i=1,\ldots,r$, and $d=\deg(\cL)-3m>\sum_{i=1}^r \mu_i$, which
is indeed equivalent to $(\PP^2,\cL)$ being admissible.


Now we prove the following result:

\begin{theorem}\label{thm:deg}
Let $(\Pl,\cL)$ be an admissible plane model.
Let $\gamma\colon\Pl\rto\Pl$ be a Cremona transformation such that
$\deg(\gamma_*(\cL))\leqslant\deg(\cL)$.
Then there exists a de Jonqui\`eres transformation $\phi\colon\Pl\rto\Pl$
such that $\deg(\phi_*(\cL))\leqslant \deg(\gamma_*(\cL))$.
\end{theorem}

\begin{proof}[Beginning of proof.]
By Noether-Castelnuovo Theorem \ref{thm:NC}, $\gamma$ is the composition of
finitely many quadratic and linear transformations. We may and will assume
that the number of involved quadratic transformations is minimal.

Let $p$ be the maximal multiplicity point of $\cL$
and $p_i$, $i=1,\ldots,r$, the points of multiplicity $\nu_i>m=[(\deg(\cL)-\mult_p(\cL))/2]$,
so $\ad_m(\cL)$ is given by \eqref{eq:adm}.

Let $\Lambda=\cL(\delta;\alpha,\alpha_1,\ldots,\alpha_r,\beta_1,\ldots,\beta_s)$
be the homaloidal net defining $\gamma$, where $\alpha_i$, $i=1,\ldots,r$, [$\alpha$, resp.]
is the multiplicity of $\Lambda$ at $p_i$ [at $p$, resp.],
and $\beta_j$, $j=1,\ldots,s$, is the multiplicity at the other base points
$q_1,\ldots,q_s$ of $\Lambda$.

Now we may assume $k=2m$, otherwise we substitute $\cL$ with its double.

The linear system $\ad_m(\cL)$ has $b=\sum_{i=1}^r \mu_i$ fixed lines,
say $L_1,\ldots,L_t$, $t\leqslant r$, each $L_i$ counted with multiplicity $b_i$
with $b=\sum_{j=1}^t b_j$. The movable part $\cL(d-b;d-b)$ of $\ad_m(\cL)$ is composed with the pencil $\cL_0=\cL(1;1)$ of lines through $p$.
Our hypothesis says that $\deg(\ad_m(\gamma_*(\cL)))\leqslant\ad_m(\cL)$.
This implies that, either
\begin{enumerate}[(1)]
\item $\gamma$ contracts a fixed line to a point; or
\item $\gamma$ does not contract any fixed line and therefore $\gamma$ maps $\cL_0$
to a pencil of lines, in which case $\deg(\phi_*(\cL)))=\deg(\cL)$
and $\gamma$ is a de Jonqui\`eres.
\end{enumerate}
In case (2),
there is nothing else to prove.
So we may and will assume that case (1) holds and $\gamma$ is not a de Jonqui\`eres.
We will then argue by induction on the simplicity of $\gamma$.
This ends the first part of the proof.
\end{proof}

Next we need the following lemma:

\begin{lemma}
In the above setting, there is a fixed line of $\ad_m(\cL)$, say $L_1$,
contracted by $\gamma$, passing through $p$ and an infinitely near point $p_1$ to $p$,
such that $\delta=\alpha+\alpha_1$, with $\alpha\geqslant\alpha_1\geqslant \alpha_i,\beta_j$,
$i=1,\ldots,r$, $j=1,\ldots,s$.
\end{lemma}

\begin{proof}
We may and will assume that $L_i$, $i=1,\ldots,t$, is the line passing through $p$
and $p_i\infnear[1] p$.
If there is an $i$, $1\leqslant i \leqslant t$, such that $\alpha_i=\delta-\alpha$,
we may assume that $i=1$ and the rest of the assertion follows.
Suppose, by contradiction, that $\alpha_i<\delta-\alpha$ for each $i=1,\ldots,t$.

The Cremona transformation $\gamma$ maps the pencil $\cL_0$ to a pencil of curves
of degree $\ell=\deg(\gamma_*(\cL_0))=\delta-\alpha\geqslant2$
and the line $L_i$, $i=1,\ldots,t$, to a curve of degree
$\ell_i=\deg(\gamma_*(L_i))\leqslant\delta-\alpha-\alpha_i$.
Then, either $\ell_i>0$, or $\ell_i=0$, which means that
there is at least a proper or infinitely near point $q_{j_i}$ on $L_i$
or on its proper transform.
By \eqref{eq:phi!}, one has
\begin{equation}\label{ad_phi_B}
\deg(\ad_m(\gamma_*(\cL)))=\deg(\gamma_!(\ad_m(\cL)))
-d\alpha - \sum_{i=1}^r \mu_i\alpha_i  +\sum_{j=1}^{s} h_j\beta_{j} ,
\end{equation}
where $0\leqslant h_j=m-h'_j$ where $h'_j$ is the multiplicity of $\cL$ at $q_j$, $j=1,\ldots,s$,
and the total transform $\gamma_!(\ad_m(\cL))$ of $\ad_m(\cL)$ via $\gamma$ has degree
\begin{equation}\label{eq:phi!ad_B}
\deg(\gamma_!(\ad_m(\cL)))
=(d-b)\ell+\sum_{i=1}^t b_i \ell_i+d\alpha+
\sum_{i=1}^t b_i\alpha_i+\sum_{i=1}^t \sum_{q_j\in L_i} b_i\beta_j,
\end{equation}
where the last sum runs over the proper or infinitely near base points $q_j$ of $\Lambda$
which belong to the proper transform of $L_i$
(there is no $p_i$, $i=t+1,\ldots,r$, among such points,
otherwise $L_i$ shoud be a component of $B$).
Therefore, \eqref{ad_phi_B} and \eqref{eq:phi!ad_B} imply
\[
\deg(\ad_m(\phi_*(\cL)))\geqslant (d-b)\ell
+\sum_{i=1}^t b_i \Biggl(\ell_i+\sum_{q_j\in L_i} b_i\beta_j\Biggr)
\geqslant 2d-2b+\sum_{i=1}^t b_i=d+(d-b)>d,
\]
which contradicts the assumption that $\deg(\gamma_*(\cL))\leqslant\deg(\cL)$.
\end{proof}

\begin{proof}[Proof of Theorem \ref{thm:deg} (continued).]
By the proof of Noether-Castelnuovo Theorem \ref{thm:NC}
in Appendix \ref{app:proof}, there is a quadratic
transformation $\psi\colon\Pl\rto\Pl$ centered at $p$ such that $\gamma=\gamma'\circ\psi$
with $\gamma'$ simpler than $\gamma$.
Set $\cL'=\psi_*(\cL)$.
If $\deg(\cL')\leqslant\deg(\cL)$, then the minimality assumption on $\gamma$ implies that
$\gamma$ is quadratic, a contradiction because we are assuming that $\gamma$
is not a de Jonqui\`eres.
Therefore $\deg(\cL')>\deg(\cL)\geqslant\deg(\gamma_*(\cL))=\deg(\gamma'_*(\cL'))$.
Furthermore, the number of quadratic transformations involved in the factorization of $\gamma'$ 
is minimal, as well as the analogous number for $\gamma$.

In Lemma \ref{lem:psi} below,
we prove that we may choose the quadratic transformation $\psi$ in such a way that $\ad_m(\cL')=\cL(d';(d',\{\mu'_1,\ldots,\mu'_r\}))$ and $\cL'$ has maximal multiplicity at $p'$,
namely $\psi$ maps the pencil $\cL_0$ of lines through $p$ to the pencil of lines through $p'$.
By induction, there is a de Jonqui\`eres transformation $\varphi$ centered at $p'$ such that $\deg(\varphi_*(\cL'))\leqslant\deg(\gamma'_*(\cL'))=\deg(\gamma_*(\cL))$.
Therefore $\phi:=\varphi\circ\psi$ is a de Jonqui\`eres transformation centered at $p$
such that $\deg(\phi_*(\cL))\leqslant\deg(\gamma_*(\cL))$, which concludes the proof.
\end{proof}

\begin{lemma}\label{lem:psi}
In the above setting, we may choose the quadratic transformation $\psi$ in such a way that
$(\Pl,\cL')$ is admissible, namely $\ad_m(\cL')=\cL(d';(d',\{\mu'_1,\ldots,\mu'_r\}))$.
\end{lemma}

\begin{proof} The simplicity of $\Lambda$
is $(k_\Lambda, h_\Lambda, s_\Lambda)$, where $k_\Lambda=\delta-\alpha=\alpha_1$.
Since $p_1$ is a point of multiplicity $\alpha_1> \frac {k_\Lambda}2$, the proof of Noether--Castelnuovo's Theorem in Appendix A 
implies that there is a quadratic transformation $\psi$  based at $p$,
at $q\infnear[1]p$ and at $x$,
which lowers the simplicity of $\Lambda$.   Let $\cL'=\psi_*(\cL)$,  which has maximal multiplicity at $p'$, the point corresponding to $p$ via $\psi$.  
We have to prove that $\cL'$ does not have base points of multiplicity $\mu>m$ off $p'$.  

Assume first that $x$ is a proper point, hence the multiplicity of $\cL$ at $x$ is $\nu\leqslant m$. Then a point
of multiplicity  $\mu>m$  for $\cL'$ off $p'$ could come only
from a point $y\infnear[1]q$, $y\notsatel p$. But in that case the quadratic transformation based at $p, q, y$ lowers the degree of $\cL$, a contradiction.

The alternative is that $x\infnear[1]q$.
As above, the multiplicity of $\cL$ at $x$ is $\nu\leqslant m$ and therefore
$\psi$ cannot produce points of multiplicity $\mu>m$ off $p'$.  
\end{proof}

\section{Weight of a curve on a $\F_n$ and good plane models}\label{S:sing}

Here we introduce the notions needed for finding Cremona minimal pairs.

Let $(\F_n,C)$ be a pair where $n\geqslant2$ and $C$ is a singular, irreducible curve,
and let $p\notin E$ be a singular point of $C$.
The \emph{weighted oriented tree}, or briefly \emph{tree},
$T_p$ of $p\in C\subset\F_n$
is defined as follows:
\begin{enumerate}[Step 1:]

\item the \emph{root} $v_p$ of $T_p$ corresponds to $p$
with weight the multiplicity of $C$ at $p$;

\item make $\elm_p$ and let $q=\elm_p(F_p)\in E$;
consider the proper transform $C'$ of $C$
and the singular points $p_1,\ldots,p_r$ of $C'$
lying on $F_q$ off $E$;
then $T_p$ has one vertex $v_{p_i}$, with weight the multiplicity of $C'$ at $p_i$,
and one arrow $v_p\to v_{p_i}$,  for each $i=1,\ldots,r$;

\item iterate Step 2 for $p_1,\ldots,p_r$ until either no further
singular point of $C'$ shows up or the maximal length
of oriented chains in the tree is $n-1$.
\end{enumerate}

The \emph{weighted oriented forest},
or briefly \emph{forest},  $G=G(\F_n,C)$ of the pair $(\F_n,C)$
is constructed as follows:
\begin{enumerate}[$\ \bullet\ $]

\item if $F_p$ is the fibre of the ruling $|F|$
containing a singular point $p$ of $C$ off $E$,
then $G$ has a vertex $v_{F_p}$, with zero weight;

\item for each singular point $p\notin E$ of $C$,
then $G$ contains the tree $T_p$
of $p\in C$, connected to $v_{F_p}$
with an arrow $v_{F_p}\to v_p$;

\item $G$ has further vertices $v_{p_i}$, with weight 1,
and $v_{F_{p_i}}$, with zero weight, and an arrow $v_{F_{p_i}}\to v_{p_i}$,
for each $i=1,\ldots,n-1$,
where $p_1,\ldots,p_{n-1}$ are general points of $C$.
\end{enumerate}

A  \emph{path} $P$ in $G$  is the union of oriented chains
in distinct connected components of $G$,
each chain starting from a vertex of weight 0.
The \emph{length} of $P$ is the number of arrows that it contains.
Given a path of lenght $n-1$, the 
set $\{p_1,\ldots,p_{n-1}\}$ of corresponding points is called a \emph{cluster}
for the pair $(\F_n,C)$.
The \emph{weight} of $P$ is the sum  of the weights
of all vertices where $P$ is supported.
A path $P$ is called \emph{good}
if $P$ has length $n-1$ and maximal weight.
Accordingly, the cluster $\{ p_1,\ldots,p_{n-1} \}$ of corresponding points
is called \emph{good}. Setting $m_i$, $i=1,\ldots,n-1$, the multiplicity of $C$
at the infinitely near or proper point $p_i$, which is the weight of $G$ at the
corresponding vertex,
we may and will assume that $m_1\geqslant m_2\geqslant \cdots \geqslant m_{n-1}\geqslant 1$
and we say that $m_1,\ldots,m_{n-1}$
is a \emph{good sequence of multiplicities} for the pair $(\F_n,C)$.
The \emph{weight} $w(\F_n,C)$ of the pair $(\F_n,C)$
is the weight $w(P)=\sum_{i=1}^{n-1} m_i$ of a good path $P$
in the forest $G$.

\begin{remarks}
The number of connected components of $G=G(\F_n,C)$ equals
the number of fibres of $|F|$ contaning singular points of $C$
off $E$, increased by $n-1$.
On each connected component of $G$ there is a unique vertex of weight 0.

There exist good paths in $G$.
For example, if $n=2$, a good path is just an arrow
$v_{F_p}\to v_p$ where $p$ is a point of maximal multiplicity of $C$
off $E_2$.
\end{remarks}

\begin{examples}\label{ex:forest}
$(a)$
Let $(\F_n,C)$ be a pair with $n\geqslant2$ and $C$ smooth.
Then $1^{n-1}$ is the unique good sequence of multiplicities of $(\F_n,C)$
and any set of $n-1$ general points of $C$ is a good  cluster.

\medskip
\noindent
$(b)$
Let $(\F_n,C)$ be a pair with $n\geqslant2$ and $C$ irreducible.
Suppose that the largest $n-1$ multiplicities
$m_1\geqslant m_2\geqslant \cdots \geqslant m_{n-1}$ of $C$ can be found, respectively,
at points $p_1,\ldots,p_{n-1}$ on distinct fibres, off $E_n$.
Then $w(\F_n,C)=\sum_{i=1}^{n-1} m_i$
and $m_1,\ldots,m_{n-1}$ is a good sequence of multiplicities of $C$.
In particular, if the singularities of $C$
are only at proper points on distinct fibres, off $E_n$,
then there is a unique good sequence of multiplicities of $C$, but
perhaps different good clusters. 

\medskip
\noindent
$(c)$ Let $(\F_3,C)$, $C\in\cL_3(6,18;3,(2,[2]))$, i.e.\ $C$ has
a triple point $p$ and a tacnode $p'$. Suppose that $p,p'$
lie on the same fibre $F_p=F_{p'}$ off $E_3$ and that the tacnodal tangent at $p'$
is different from the tangent line to $F_p$ at $p'$. Then the weighted oriented
forest $G=G(\F_3,C)$ of $(\F_3,C)$ is
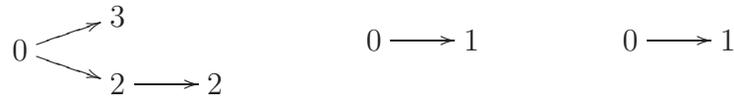
\begin{align*}
 & \,
&&  \xymatrix@R-5ex{ & 3 \\ 0 \ar [ur] \ar[dr] \\  & 2 \ar[r] & 2 }
&&  \xymatrix@R-5ex{ \\ 0 \ar[r] & 1 }
&&  \xymatrix@R-5ex{ \\ 0 \ar[r] & 1 }
&& \,
\end{align*}
that has two different good sequences of multiplicities $2,2$ and $3,1$,
determined by the different good paths: $0\rightarrow2\rightarrow2$
and $\{0\rightarrow3\}\cup \{0\rightarrow1\}$.

Accordingly, one gets two birationally equivalent different
Cremona minimal pairs
$(\Pl,B_1)$ and $(\Pl,B_2)$ with $B_1\in\cL(14;(8,[4^2]),3)$
and $B_2\in\cL(14;(8,[5,3]),2^2)$, where all singularities of $B_1, B_2$
are infinitely
near to the point of multiplicity 8 and their configuration is described respectively by the following two diagrams 
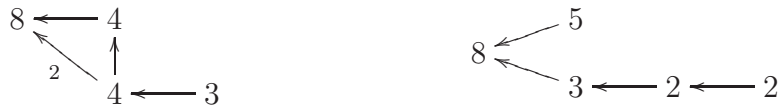
\begin{align*}
 & \xymatrix@R-2ex{ 8 & 4 \ar[l] \\  & 4 \ar[ul]^{2} \ar[u] & 3 \ar[l] }
&& \xymatrix@R-5ex{ & 5 \ar[dl] \\ 8 \\ & 3 \ar[ul] & 2 \ar[l] & 2 \ar[l] }
\end{align*}
where $m_1 \xrightarrow{\ r\ } m_2$ means that the point of multiplicity $m_1$
is proximate and infinitely near of order $r$ to the point of multiplicity $m_2$
(no $r$ means $r=1$).

This example $(c)$ shows that a good sequence of multiplicities may well not be unique.
\end{examples}

Now we study the properties of good sequences of multiplicities
of $\natural$-models obtained from $\sharp$-models.

\begin{remarks}\label{rem:number>m}
$(i)$
Let $(\F_{n'},C')$ be a $\flat$-model where $n'\geqslant1$,
$C'\in\cL_{n'}(2m,h';m_1,\ldots)$ and $C'$ irreducible.
The $\natural$-model $(\F_n,C)$ fibred
birationally equivalent to $(\F_{n'},C')$
is such that $C$ has at most $n-n'$ proper or infinitely
near points of multiplicity $\mu>m$.

\medskip
\noindent
$(ii)$
Let $m'_1\geqslant\cdots\geqslant m'_{n-1}=\bar m$ be a good sequence
of multiplicities for the $\natural$-model $(\F_n,C)$ as above
and let $\{p_1,\ldots,p_{n-1}\}$ be the corresponding good cluster.

The proper transform $C_2=(\elm_{p_1,\ldots,p_{n-2}})_*(C)\subset \F_2$
has multiplicity $\bar m$ at $p_{n-1}\notin E_2$.
The very definition of good cluster implies that
$C_2$ has no point of multiplicity $\mu >\bar m$ off $E_2$.
If $(\Pl,B)$ is the good model of $(\F_n,C)$ and $q\in\Pl$
the point of $B$ with the highest multiplicity,
it follows that each point of $B$
with multiplicity $\mu>\bar m$, if any,
is infinitely near to $q$.

\medskip
\noindent
$(iii)$
If $\bar m\geqslant m$, then $C$ has multiplicity $\mu\geqslant m$
at each one of the $n-1$ points $p_1$, \ldots, $p_{n-1}$ of the good cluster,
hence $n'=1$ by (i) and $C_1=(\elm_{p_1,\ldots,p_{n-1}})_*(C)=C'$ on $\F_1$.
Thus $B=\sigma_*(C_1)\in\cL(h';h'-2m,m_1,\ldots)$, with $m_1\leqslant m$,
and the good plane model $(\Pl,B)$ is Cremona minimal by Theorem \ref{thm:jung}.
In particular, if $n'\geqslant2$, it follows that $\bar m< m$.

\medskip
\noindent
$(iv)$
Suppose that $B$ has multiplicity $m_1>m$ at a point $q_1>^1 q$
and let $L_1$ be the line passing through $q$ and $q_1$.
If $B$ meets $L_1$ at a proper or infinitely near point $q'$,
different from $q$ and $q_1$, then $B$ has multiplicity $\mu <m$ at $q'$,
because
\[
\deg(B)-\mult_{q}(B)-\mult_{q_1}(B)=2m-m_1< m.
\]
\end{remarks}


\begin{theorem}\label{lem:new} Let $(\Pl,B)$ be a good model 
of a $\natural$-model  $(\F_n,C)$, $n\geqslant 2$ with $\flat$--index  at least $2$, $C$ irreducible,
$C\in \cL_n(2m+\epsilon,h;m_1,\ldots)$, and $\epsilon=0,1$. Then 
the pair $(\Pl,B)$ is Cremona minimal.
\end{theorem}

\begin{proof} Let $p_0$ be the point of maximal multiplicity of $B$.
Suppose that there is a Cremona transformation $\gamma$ such that
$\deg(\gamma_*(B))<\deg (B)$ and $\deg(\gamma_*(B))$ is minimum.
Since $(\PP^2, B)$ is admissible, 
by Theorem \ref{thm:deg} there is a de Jonqui\`eres transformation $\phi$
centered at $p_0$
such that $\deg(\phi_*(B))=\deg(\gamma_*(B))<\deg(B)$,
and moreover $\phi$ maps the pencil of lines through the maximal multiplicity
point of $B$ to the pencil of lines through the maximal multiplicity
point of $B'=\phi_*(B)$.

The curve $B$ sits in a linear system of the form $\cL(d+3m;(d+m-\epsilon, [2m+\epsilon-\mu_1,\ldots,2m+\epsilon-\mu_{n-1}], m'_1,\ldots)$, where $d=\deg({\rm ad}_m(B))$,  the $\natural$--model is $(\F_n,C)$ and $\mu_1,\ldots,\mu_{n-1}$ is a good sequence of multiplicities of $C$ (not necessarily in decreasing order). Let $\Lambda=\cL(\delta;\delta-1,1^{2\delta-2})$ be the homaloidal net
defining  $\phi$. Let  $\ell\leqslant n-1$ be the number of points among those of multiplicities $2m-\mu_1,\ldots,2m-\mu_{n-1}$ for $B$, which are simple  base points for $\Lambda$. By proximity, we have $\ell\leqslant \delta-1$. Then
$$\deg(\phi_*(B))=\delta(d+3m)-(\delta-1)(d+m-\epsilon)-\sum_{i=1}^\ell (2m+\epsilon-\mu_i)-\sum _{\ell+1}^{2\delta-2}\mu'_j$$
where $\mu'_j$ are the multiplicities of $B$ at the remaining base points of $\Lambda$. Therefore one has
$$deg(\phi_*(B))=2m(\delta-\ell)+d+m+\epsilon(\delta-1-\ell)+\sum_{i=1}^\ell \mu_i-\sum_{i=\ell+1}^{2\ell} \mu'_i-
\sum_{i=2\ell+1}^{2\delta-2} \mu'_i$$
where we assume that the $ \mu'_i$ are ordered in decreasing order for $\ell+1\leqslant i\leqslant 2\delta-2$.
Note that the general curve of the proper transform of $\Lambda$ on $\F_n$ is a smooth section. Hence
all the base points of this proper transform are parts of paths in the relevant forest. Therefore 
$\sum_{i=1}^\ell \mu_i\geqslant \sum_{i=\ell+1}^{2\ell} \mu'_i$. Moreover we have $ \mu'_i\leqslant m$ for $2\ell+1\leqslant i\leqslant 2\delta-2$ as a consequence of the assumption on the $\flat$--index (see Remark \ref{rem:number>m}, (iii)). Hence, we deduce
$$deg(\phi_*(B))\geqslant 2m(\delta-\ell)+d+m-2(\delta-\ell-1)m=d+3m$$
a contradiction. \end{proof}

\begin{remark}\label{rem:equal}
The above proof shows that one may have $\deg(\phi_*(B))=\deg (B)$ for a de Jonqui\`eres transformation only if $\sum_{i=1}^l \mu_i=\sum_{i=\ell+1}^{2\ell} \mu'_i$,  $\mu'_i=m$ for 
$2\ell+1\leqslant i\leqslant 2\delta-2$ and, in addition, $\ell=\delta-1$ if $\epsilon=1$.
\end{remark}


\section{The main classification theorem}\label{S:theorem}

Let $(S,C)$ be a pair. We will say that it presents the \emph{line case}
if it is birationally equivalent to $(\Pl,L)$,
where $L$ is a line.
For example, if $C$ is part of a $(-1)$-cycle on $S$,
then $(S,C)$ presents the line case.
In \cite{Coolidge}, Coolidge stated the following theorem,
which gives necessary and sufficient
conditions for a pair $(S,C)$ to present the line case:

\begin{theorem}[Coolidge] \label{thm:Coolidge}
Let $B$ be an irreducible plane curve and let $C\subset S$ be the minimal
resolution of $B$, such as  in  \cite[Proposition V.3.8]{Hartshorne}.
Then $(\Pl,B)$ presents the line case
if and only if $\vert C+mK_S\vert=\emptyset$, for all $m>0$.
\end{theorem}

Unfortunately, Coolidge's proof in \cite[p.\ 396--398]{Coolidge},
written again in \cite[p.\ 772--773]{KumarMurthy}, is incomplete.
Kumar and Murthy gave a correct proof of Theorem \ref{thm:Coolidge} with different methods.
They actually proved more (see \cite{KumarMurthy}, Theorem 2.1
and Corollary 2.4):

\begin{theorem}[Kumar-Murthy] \label{thm:KM} 
Let $B$ and $(S,C)$ be as in Theorem \ref{thm:Coolidge}.
Then $(\Pl,B)$ presents the line case if, and only if,
one has $\vert C+K_S\vert=\vert C+2K_S\vert=\emptyset$ or, equivalently, if and only if $\kappa(S,C)=-\infty$.
In particular, if moreover $C^2=-n$, then $(\Pl,B)\sim(\F_n,E)$.
\end{theorem}

\begin{remark}
In the above setting, if  $C$ is
rational and $C^2\geqslant -3$, then $\vert C+K_S\vert=\vert C+2K_S\vert=\emptyset$
and $(S,C)$ presents the line case. Hence, if $C$ is irreducible and the pair
$(S,C)$ does not present the line case, then $C$ verifies the hypothesis $(ii)$
of Proposition \ref{prop:zariski}. 
\end{remark}

Next we will prove a birational classification theorem 
for all pairs $(S,C)$,with $S$ rational, producing for each of them a unique
model on a minimal rational surface or on $\F_1$. Moreover we will produce
for each pair a plane model 
of minimal Cremona degree. 

First we introduce some additional notation. Given a pair $(S,C)$, with 
$S$ smooth and rational and $C$ smooth, not presenting the line case, let

\centerline{$m:=m(S,C)$ be the minimum positive integer $m$
such that}
\centerline{ $\vert C+mK_S\vert\ne\emptyset$ and 
$\vert C+(m+1)K_S\vert=\emptyset$}
\noindent and
$$\alpha:=\alpha(S,C)=\dim (\vert C+mK_S\vert)=h^0(S,\cO_S(C+mK_S))-1.$$

By the results in \S  \ref{susec:birat}, $m$ and $\alpha$ are birational invariants
of the pair $(S,C)$.
Note that, if $C$ is rational, then $m\geqslant 2$
and moreover $C$ is not part of a $(-1)$-cycle on $S$.

Our main results are the following theorems.

\begin{theorem}[Birational classification of pairs] \label{thm1}
Let $(S,C)$ be a pair with $S$ rational and $C$ smooth and irreducible,
and suppose that $(S,C)$ does not present the line case. Let 
$m=m(S,C)$ and $\alpha=\alpha(S,C)$. 
Then $(S,C)$ is birationally
equivalent to one of the following pairs:
\begin{enumerate}[(i)]

\item [$(dp_1)$] $(\Pl,D)$, where
$D\in\cL(3m; m_0, \ldots)$, with $m_0\leqslant m$, and $\alpha=0$;
\label{DP1}

\item [$(dp_2)$] $(\F_n,D)$, where
$D\in\cL_n( 2m,(2+n)m; m_1, \ldots)$
with $m_1< m$, $n=0,2$, and $\alpha=0$;
\label{DP2}

\item [$(r)$] $(\F_n,D)$, $0\leqslant n\leqslant 2+\left[\,{\alpha}/{m} \right]$,
where
$
D\in\cL_n(2m,(2+n)m+\alpha;m_{1}, \ldots),
$
is $\flat$--minimal, i.e.\ $m_{1}\leqslant m$, $\alpha> 0$, and,  if $n\geqslant 1$, all singular points of multiplicity $m$
are on  $E_n$;

\item [$(b_1)$] $(\Pl,D)$, where
$D\in\cL(3m+\left[{\alpha}/{2}\right]; m_0, \ldots)$,
with $m_0\leqslant m$, $\alpha=2,5$;

\item [$(b_2)$] $(\F_n,D)$, 
where
$
D\in\cL_n(2m+1,(2+n)m+({\alpha+n-1})/2;
m_{1}, \ldots),
$
is $\flat$--minimal, with $m_{1}\leqslant m$, $\alpha+n\equiv 1,  ({\rm mod} \ 2)$, and
$\alpha \geqslant 3+n$ if $0\leqslant n\leqslant 1$, whereas 
$\alpha\geqslant 3+(n-2)(2m+1)\geqslant 3$ if $n\geqslant 2$.
\end{enumerate}

The above pairs may be birationally equivalent and not isomorphic  only if we are in case:

\begin{enumerate}[(i)]
\item $(dp_1)$, with $m_0=m_1=m_2=m$;
\item $(r)$,  with $D$ having at least two  points of multiplicity $m$. 
\end{enumerate}
\end{theorem}

\begin{remark} By taking into account the proof of Proposition \ref{pro:flat}, one sees that part (ii) of the statement above can be improved. Indeed, $D$ may have more points of multiplicity $m$, provided 
performing elementary transformations there forces to make the inverse transformations to go back to 
the $\flat$--model. This can be expressed in terms of clusters on the $\natural$--model, but we will not dwell on this here. \end{remark}

\begin{remark}\label{rem:sharp}
All pairs in the statement of Theorem \ref{thm1},
but types $(dp_1)$ and $(b_1)$, are $\sharp$-models
(see \S \ref{ssec:minimalrat}).
By blowing up a point of multiplicity $m_0$,
one sees they
are respectively birationally equivalent to
the $\sharp$-models:
\begin{align}
\label{dp1F1}
& (\F_1,D'),
\quad
D'\in\cL_1(3m-m_0,3m;m_1,\ldots),
\quad &&\text{with } m_1\leqslant m_0\leqslant m, \\
\label{b1F1}
& (\F_1,D'),
\quad
D'\in\cL_1(3m-m_0+[\alpha/2],3m+[\alpha/2];m_1,\ldots),
\quad &&\text{with } m_1\leqslant m_0\leqslant m.
\end{align}
 
Pairs $(dp_2)$,  $(r)$ if either $n=0$ or $m>m_1$, $(b_1)$ and \eqref{b1F1}
are $\sharp\sharp$-models.
Pair \eqref{dp1F1} is a $\sharp\sharp$-model if either $m_0<m$
or $m=m_0>m_1$. 
Pairs $(r)$ with $n\geqslant 1$ have positive $\flat$--index.
\end{remark}

In cases $(r)$ and $(b_2)$, with $n\geqslant 2$,  since the
$\sharp$--index is positive, 
we may consider the unique $\natural$-model $(\F_{n+v},D_\natural)$ which is
fibred birationally equivalent to $(\F_n,D)$ (see Proposition \ref{prop:percent}).
Let $q_1,\ldots,q_{v}$ be the proper or infinitely near singular points of $D$
lying on $E_n$ or on its strict transform, 
and let $\mu_1, \ldots,\mu_{v}$ be the respective multiplicities,
where we may assume
$\mu_1\geqslant \ldots \geqslant \mu_{v}\geqslant 2$. Set $\gamma=\sum_{i=1}^{v}(2m-\mu_i)$ and $n'=n+v$.
Then $(\F_{n+v},D_\natural)$ is obtained by performing
elementary transformations at $q_{1},\ldots,q_{v}$ and therefore
either $D_{\natural}\in \cL_{n'}(2m,(2+n)m+\alpha+\gamma)$, in case $(r)$ with $n\geqslant 1$,
or $D_{\natural}\in \cL_{n'}(2m+1,(2+n)m+(\alpha+n-1)/2+\gamma)$, in case $(b_2)$ with $n\geqslant 1$.

Let $m'_1,m'_2,\ldots,m'_{n'-1}$ be a good sequence of multiplicities
of the $\natural$-model $(\F_{n'},D_\natural)$,
and set $\beta=\sum_{i=1}^{n'-1 }m'_i$. With this notation, we have:

\begin{theorem}[Cremona minimal pairs] \label{thm2}
Let $(S,C)$ be a pair with $S$ rational and $C$ smooth and irreducible,
and suppose that $(S,C)$ does not present the line case.
Then $(S,C)$ is birationally
equivalent to a Cremona minimal pair $(\Pl,B)$,
where $B$ belongs to one of the following planar linear systems (where $m$ and $\alpha$ are the birational invariants introduced above):
\begin{enumerate}

\item [$(cdp_1)$] $\cL(3m; m_0, \ldots)$,
with $m_0\leqslant m$;
\label{CDP1}

\item [$(cdp_2)$]$\cL(4m-m_1;2m-m_1,2m-m_1, m_2, \ldots)$
with $m_2\leqslant m_1<m$, $m_1>0$;
\label{CDP2}

\item [$(cdp_3)$]$\cL(4m-m_1;2m-m_1,[2m-m_1], m_2, \ldots)$
with $m_2\leqslant m_1<m$, $m_1>0$;
\label{CDP3}

\item [$(cr_0)$]$\cL(4m-m_1+\alpha;2m-m_1+\alpha,2m-m_1, m_2, \ldots)$,
with $m_2\leqslant m_1<m$, $m_1>0$  and $\alpha> 0$;

\item [$(cr_1)$]$\cL(3m+\alpha;m+\alpha, m_1,\ldots)$, with $m_1\leqslant m$ 
and $\alpha> 0$;

\item [$(cr_2)$] $\cL((2+n)m-\beta+\alpha+\gamma;
(nm-\beta+\alpha+\gamma,[2m-m'_{n'-1}, 2m-m'_{n'-2}, \ldots, 2m-m'_1]),
m_{n'+1}, \ldots)$, with with $2\leqslant n\leqslant 2+[\alpha/m]$, $\alpha> 0$, 
$m\geqslant m_{n'+1}\geqslant \dots$
and $n'$, $\gamma$, $m'_1,m'_2,\ldots,m'_{n' -1}$ and $\beta$ as above;

\item [$(cb_1)$]  $\cL(3m+\left[{\alpha}/{2}\right]; m_0, \ldots)$,
with $m_0\leqslant m$ and $\alpha=2,5$;

\item [$(cb_2)$] $\cL(3m+{\alpha}/{2}; m-1+{\alpha}/{2}, m_1, \ldots)$,
with $m_1\leqslant m$ and $\alpha\geqslant 4$ is even;

\item [$(cb_3)$] $\cL(4m-m_1+({\alpha+1})/{2};
2m-m_1+({\alpha-1})/{2},2m+1-m_1, m_2, \ldots)$,
with $m_2\leqslant m_1\leqslant m$ and $\alpha\geqslant 3$ is odd;

\item [$(cb_4)$] $\cL((2+n)m-\beta+\gamma+({\alpha+n-1})/{2};
nm-\beta+\gamma+({\alpha+n-3})/{2},
[2m+1-m'_{n'-1}, 2m+1-m'_{n'-2}, \ldots, 2m+1-m'_1],
m_{n'+1}, \ldots )$,
with $2\leqslant n\leqslant 2+({\alpha-3})/({2m+1})$,
$m_{n'+1}\leqslant m$, 
and $n'$, $\gamma$, $m'_1,m'_2,\ldots,m'_{n' -1}$ and $\beta$
are as above.
\label{cb4}
\end{enumerate}

The above pairs may be Cremona,  but not projectively, equivalent only if we are in case:

\begin{enumerate}[(1)]
\item  $(cdp_1)$  with $ m_0=m_1=m_2=m$ and and $(cr_1)$ with $m_1=m_2=m$;
\item  $(cdp_2)$,  $(cdp_3)$ and $(cr_0)$ with $m_1=m_2$;
\item  $(cr_2)$, for different good sequences of multiplicities of the same $\natural$--model;
\item  $(cb_4)$, for different good sequences of multiplicities of the same $\natural$--model. 
\end{enumerate}

Except in cases (3) and (4), the Cremona type of the minimal Cremona pairs in a Cremona equivalence class is unique. 
\end{theorem}

\begin{remark}
All systems in Theorem \ref{thm2}, but $(cr_2)$ and $(db_4)$,
are of Noether type, and therefore, by Theorem \ref{thm:jung}, their Cremona minimality
is clear. 
\end{remark}

The proof of  Theorems \ref{thm1}  and \ref{thm2} will follow
from the analysis of three different cases, according to the behaviour
of the nef part of $C+mK_S$. Indeed,
by  Propositions \ref{prop:zariski} and \ref{prop:zariski2},
one has $C+mK_S\equiv P+N$, where $P$ is the nef and $N$ is the negative part
of $C+mK_S$ as in \eqref{eq:C+mK}.
Then there are three possible cases: either
\begin{enumerate}[$(A)$]
\item $P=0$; or
\item $P>0$ and $P^2=0$; or
\item $P>0$ and $P^2>0$.
\end{enumerate}

We call cases (A), (B), and (C) respectively the \emph{Del Pezzo},
\emph{ruled}, and \emph{big case},
and we will discuss them separately.

\subsection{The Del Pezzo case}

This case is covered by the following:

\begin{proposition}\label{prop:delpezzo} Let $(S,C)$ be a pair presenting the 
Del Pezzo case.
Then $(S,C)$ is birationally
equivalent to one, and only one, of pairs $(dp_1)$, $(dp_2)$.

Accordingly, $(S,C)$ is birationally equivalent to 
a Cremona minimal pair $(\Pl,B)$, $B\in\cL$,
where $\cL$ is one, and only one, of types $(cdp_1), (cdp_2), (cdp_3)$
in Theorem \ref{thm2}.
\end{proposition}

\begin{proof}
Let $f\colon S\to S'$ be the birational morphism
which first blows down the negative part $N$ of $\vert C+mK_S\vert$  
and then other $(-1)$-cycles in such a way that $S'$ is minimal.
On $S'$, one has $C'=f_*(C)\equiv -mK_{S'}$
and a $(-1)$-cycle $\theta$ on $S$, which is blown down by $f$
and is not part of $N$, is such that $C\cdot \theta=m$.
If $S'=\F_n$, one has $0\leqslant C' \cdot E_n=(2-n)m$, therefore $n\leqslant2$.
Thus, there are three sub-cases:
either $S'=\Pl$, or $S'=\F_0$, or $S'=\F_2$.

We may and will assume that, if $S'=\F_0$ or $\F_2$,
the birational morphism $S\to S'$
does not factor through a birational morphism $g\colon S\to \F_1$
which blows down $N$.
Indeed, if such a morphism $g$ exists, we may assume that $S'=\Pl$.

We discuss separately the three cases.

\begin{enumerate}[$\ \bullet\ $]

\item If $S'=\Pl$, then $C'\equiv 3mL$, where $L$ is a line,
i.e.\ $C'$ is a curve of degree $3m$
with points of multiplicity at most $m$,
hence we are in cases $(dp_1)$ and
$(cdp_1)$.

\item If $S'=\F_0$, then $C'\equiv 2mE_0+2mF_0$
and $C'$ has points of multiplicity at most $m$.
If $C'$ had a point $p$ of multiplicity $m$,
then $\elm_p \circ f\colon S\to \F_1$ would be
a birational morphism which factors through the blowing-down of $N$,
a contradiction.
Therefore we are in case $(dp_2)$ with $n=0$. Next, choose a point $p$
of maximal multiplicity $m_1$ of $C'$. Perform an elementary
transformation $\elm_p$ and then blow down the $(-1)$-section $E_1$. 
The result is a pair $(cdp_2)$.

\item If $S'=\F_2$, then $C'\equiv 2mE_2+4mF_2$
and $C'$ has points of multiplicity at most $m$.
Note that $C'\cdot E_2=0$ implies that $C\cap E_2=\emptyset$.
The same arguments as above imply that
we are in cases $(dp_2)$, $n=2$ and $(cdp_3)$.
\end{enumerate}

The Cremona minimality of the
pairs $({cdp_1})$, $({cdp_2})$ and $({cdp_3})$
follows by Theorem \ref{thm:jung}. 

Cremona minimality implies that the pairs $({cdp_1})$ are not
Cremona equivalent to either $({cdp_2})$ or $({cdp_3})$. 
To prove that $({cdp_2})$ and $({cdp_3})$ are not Cremona equivalent, 
note that the image of $\PP^2$ via the linear system ${\rm ad}_{m_1}(\cL)$ 
is projectively different in the two cases (see Remark \ref{rem:multad}).

The assertions (i) in Theorem \ref{thm1} and (1), (2) in Theorem  \ref{thm2}  regarding the del Pezzo pairs follow form Corollary \ref{cor: jung}. \end{proof}

\begin{remark}
$(i)$ From the irreducibility of $C$,
it follows that the number of points of multiplicity $m$
of $\cL$ in case $(cdp_1)$ is at most $9$, if $m>2$
(at most $10$ if $m=2$).

\medskip
$(ii)$ An alternative proof of the birational inequivalence of pairs
$(dp_1)$ and $(dp_2)$  follows from Iitaka's Theorem \ref{thm:iitaka},
since types $(dp_2)$ and $(dp_3)$
are $\sharp\sharp$-minimal and type $(dp_1)$,
considered on $\F_1$ as in Remark \ref{rem:sharp},
is $\sharp$-minimal.
\end{remark}

\subsection{The ruled case}\label{sec:ruledcase}

We deal with this case in the following proposition.

\begin{proposition}\label{prop:ruled}
Let $(S,C)$ be a pair presenting the ruled case.
Then $(S,C)$ is birationally
equivalent to one, and only one, of the pairs in $(r)$, with the usual
exception for $n=0$.

Accordingly, $(S,C)$ is birationally equivalent to 
a Cremona minimal pair $(\Pl,B)$, $B\in\cL$,
where $\cL$ is one, and only one, of types $(cr_0), (cr_1), (cr_2)$.
\end{proposition}

\begin{proof} 
Since $\ad_{m+1}(C)=\emptyset$, one  has $\vert P+K_S\vert =\emptyset$
and Proposition \ref{D+K=0} implies that $P$ is composed
with an irreducible base point free pencil $|L|$ of rational curves,
namely $P\equiv \alpha L$.

Blowing down of all $(-1)$-cycles $Z$ such that $Z\cdot L=0$
gives a birational morphism $f\colon S\to \F_n$,
which maps $|L|$ to the ruling $|F|$ of $\F_n$ (one of the rulings if $n=0$).
Therefore
\[
D=f_*(C)\equiv -mK_{\F_n}+\alpha F\equiv 2mE+(m(2+n)+\alpha)F
\]
and $D$ has points of multiplicity at most $m$. So this is a $\sharp$--model.
We may actually assume that $n$ is the $\sharp$--index of  the pair. 

Note that $0\leqslant D\cdot E_n=2m-mn+\alpha$
implies $n\leqslant 2+\alpha/m$.

If $n=0$ we find case $(r)$ with $n=0$. 
Then we get $(cr_0)$ from $(r)$ by choosing a point $p\in D$ of
maximal multiplicity $m_1$, performing the elementary transformation $\elm_p$
and then contracting the $(-1)$-section $E_1$.

If $n=1$, we find case $(r)$ with $n=1$. 
Then we get $(cr_1)$  by contracting the curve $E_1$. 

Assume now $n>1$. We are still in case $(r)$, $n\geqslant 2$ and we get 
type $(cr_2)$ as a good model obtained from the $\natural$-model of $(\F_n,D)$,
cf.\ \S\ref{ssec:model}. 

The birational uniqueness of types in $(r)$, except for those with $n=0$ and 
at least two points of multiplicity $m$, 
follows by Remark \ref{rem:sharp}, Theorem \ref{thm:iitaka} and Proposition \ref{pro:flat}.

As for Cremona minimality, 
in cases $(cr_i)$, $0\leqslant i\leqslant 1$, one applies Theorem \ref{thm:jung}. 
In case $(cr_2)$ Cremona minimality follows from Theorem \ref{lem:new}.  

The types $(cr_i)$, $0\leqslant i\leqslant 1$ are clearly Cremona not equivalent.
The types $(cr_2)$ are not Cremona equivalent to the others, since
such a Cremona equivalence would induce a fibred birational equivalence between
the corresponding $\flat$--models (see Theorem \ref{thm:deg}), which would therefore be isomorphic by Theorem \ref{pro:flat}, a contradiction. 

Assertion (ii) in Theorem  \ref{thm1}  follows by Proposition \ref{pro:flat} in case $n>0$. If $n=0$ and there is only one point of multipliticty $m$, the assertion follows by Corollary   \ref{cor: jung}, applied to case $(cr_0)$. The assertions (1), (2) in Theorem  \ref{thm2}  regarding the ruled pairs follow again by 
Corollary   \ref{cor: jung}. As for (3) in Theorem  \ref{thm2} , it follows by Theorem \ref{lem:new} and Remark \ref{rem:equal}.  \end{proof}

\begin{remark}
Iitaka's Theorem \ref{thm:iitaka} gives
an alternative proof of the fact that types $(r_1)$ and $(r_2)$
are birationally inequivalent,
since type $(r_1)$ is $\sharp\sharp$-minimal
and type $(r_2)$ are $\sharp$-minimal.
In particular, type $(r_2)$
is not $\sharp\sharp$-minimal if $m_1=m$.
\end{remark}

\subsection{The big case}

Finally we dispose of the big case. 
This will finish the analysis of the various possible cases and 
the proof of Theorems \ref{thm1} and \ref{thm2}.

\begin{proposition}\label{prop:big}
Let $(S,C)$ be a pair presenting the big case.
Then $(S,C)$ is birationally
equivalent to one, and only one, of pairs $(b_1)$.

Accordingly, $(S,C)$
is birationally equivalent to 
a Cremona minimal pair $(\Pl,B)$, $B\in\cL$,
where $\cL$ is one, and only one, of types $(cb_1), (cb_2), (cb_3), (cb_4)$
in Theorem \ref{thm2}.
\end{proposition}

\begin{proof} Since $|C+(m+1)K_S|=\emptyset$, one has $|P+K_S|=\emptyset$
and Proposition \ref{D+K=0} implies that $|P|$ is
a base point free linear system of dimension $\alpha=P^2+1>1$.
The same argument used after Proposition \ref{D+K=0}
shows that there is a birational morphism $f\colon S\to S'$, with either
$S'=\Pl$ or $S'=\F_n$,
mapping $|P|$ to a linear system $|P'|=f_*(|P|)$,
which is one of the cases $(b^j)$, $(c_n^d)$, $n\geqslant0$,
of Theorem \ref{thm:rational}.
Furthermore, as in the Del Pezzo and in the ruled case,
the curve $C'=f_*(C)$ has points of multiplicity at most $m$.
In the $\F_n$--case we may assume $n$ is the $\sharp$--index. 

We  analyze the  different cases
which may occur.
\begin{enumerate} [$(c_0^d)$]

\item[$(b^j)$]
On $\Pl$
one has $|P'|=\cL(j)$,
so either $\alpha=\dim ( \cL(j))=2$, if $j=1$, or $\alpha=5$, if $j=2$.
Hence $C'\equiv -mK_{\Pl}+P'\equiv(3m+[\alpha/2])L$, $L$ a line,
and we are in cases $(b_1)$.

\item[$(c_1^d)$]
On $\F_1$, one has $|P'|=\cL_1(1,d)$, $d\geqslant2$,
where $d=\alpha/2$,
hence $\alpha\geqslant4$ is even, and
$
C'\equiv -mK_{\F_1}+P'
\equiv(2m+1)E+\left(3m+{\alpha}/{2}\right)F,
$
that is case $(b_2)$, $n=1$.

\item[$(c_0^d)$]
On $\F_0$, one has $|P'|=\cL_0(1,d-1)$, $d\geqslant2$,
where $d=(\alpha+1)/2$,
thus $\alpha\geqslant3$ is odd,
and
$
C'\equiv -mK_{\F_0}+P'
\equiv (2m+1)E+\left(2m+({\alpha-1})/{2}\right)F,
$
hence we are in case $(b_2)$, $n=0$.

\item[$(c_n^d)$]
On $\F_n$, $n \geqslant 2$, one has $|P'|=\cL_n(1,d)$, $d \geqslant n$,
where $d=(\alpha+n-1)/2$ and therefore
$
C' \equiv (2m+1)E+(m(2+n)+({\alpha+n-1})/{2})\,F.
$
Since
$
0\leqslant C'\cdot E_n=m(2-n)+({\alpha-n-1})/{2},
$
one has $\alpha\geqslant 3+(2m+1)(n-2)$ and we are in case $(b_2)$, $n\geqslant 2$.
\end{enumerate}

The above pairs are birationally distinct by Theorem \ref{thm:iitaka}.

Now we deal with Cremona models.
Case $(cb_1)$ and $(cb_2)$ correspond to cases $(b_1)$ and $(b)$, $n=1$. 
In case $(b)$, $n=0$, choose a point $p_1$ of the highest multiplicity $m_1$ of $C'$
and perform the elementary transformation $\elm_{p_1}\colon \F_0 \rto \F_1$;
by contracting the $(-1)$-curve $E_1$ of $\F_1$,
we arrive at case $(cb_3)$. 
In case $(b_2)$, $n\geqslant 2$, consider the $\natural$-model obtained from the $\flat$-model  we have
reached (cf.\ notation introduced in \ref{ssec:model} and used in \ref{sec:ruledcase}),
then go to a good plane model thus getting  case $(cb_4)$.

The Cremona minimality of types  $(cb_1)$, $(cb_2)$, $(cb_3)$ follows from
Theorem \ref{thm:jung}. The one of  type $(cb_4)$
can be proved with the same arguments we used in the ruled case
for $(cr_2)$. The birational inequivalence of the four cases is clear. 

The assertion (4) in Theorem  \ref{thm2}  follows by Theorem \ref{lem:new} and Remark \ref{rem:equal}. 
\end{proof}

\begin{remark}
$(i)$ Iitaka's Theorem \ref{thm:iitaka} gives
an alternative proof of the birationally inequivalence
of types $(b_1)$, $(b_2)$,
because all of them are $\sharp\sharp$-minimal
(considering $(b_1)$ as pair \eqref{b1F1} on $\F_1$,
cf.\ Remark \ref{rem:sharp}).

\medskip
$(ii)$ From Theorem \ref{thm:rational} we have that $|C+hK_S|=\emptyset$
for any $h\geqslant m$. By definition, $|C+hK_S|\ne\emptyset$ for $0\leqslant h\leqslant m$,
unless $C$ is rational, in which case $|C+K_S|=\emptyset$.

\medskip
$(iii)$ The above results can be easily extended to pairs $(S,\cL)$ with $\cL$ a positive dimensional, irreducible linear system on a rational surface $S$. Once one arrives at a $\natural$--model, the definition of a good sequence of multiplicities has to be slightly changed, inasmuch as one may perform elementary transformations only at base points of the system or at general points of the surface. We do not further dwell on this here. 
\end{remark}


\section{Applications}\label{S:applications}

\subsection{Minimality of pairs}

Let $(S,C)$ be a pair as usual.
The following definitions are due to Iitaka \cite{Iitaka1}.
The pair $(S,C)$ is called \emph{relatively minimal} if 
there is no $(-1)$-curve $E$ on $S$ such that $C\cdot E\leqslant 1$. The pair is called \emph{minimal} 
if whenever we have have a pair $(S',C')$ and a birational map $\phi: S'\dasharrow S$ inducing an equivalence $(S,C)\sim (S',C')$, then $\phi$ is a morphism. 

\begin{proposition} [Iitaka, \cite{Iitaka1}] \label{prop:Iitaka1} Let $(S,C)$  be a relatively minimal pair, with $S$ rational and $C$ irreducible. Assume $\kappa(S,C)\geqslant 0$ and either $m(S,C)>1$ or $m(S,C)=1$ and $(S,C)$ not presenting the ruled case. Then $(S,C)$  is minimal.
\end{proposition}

\begin{proof} Let $(S',C')$ be another pair and let  $\phi: S'\to S$ be a birational map inducing an equivalence $(S,C)\sim (S',C')$. 

If $m(S,C)>1$, apply Proposition \ref{prop:bir} for $m=2$. We have a diagram
 \[
\xymatrix{ & X \ar[dl]_{f'} \ar[dr]^{f} \\
S'\ar@{-->}[rr]^{\phi} & & S}
\]
Let $\Theta$ be an $f'$--exceptional $(-1)$--cycle on $X$ which  non contracted by $f$. Then $f_*(\Theta)=\theta$ is a $(-1)$--cycle on $S$ such that  $C\cdot \theta\leqslant 1$, a contradiction. This proves the assertion in this case.

If $m(S,C)=1$, then we consider $P=C+K_S$, which is nef. In the del Pezzo case, we have $C\equiv -K_S$
and relative minimality implies that there is no $(-1)$-curve on $S$,  proving the assertion. In the big case, note that relative minimality implies that there is no $(-1)$-curve $E$ such that $P\cdot E=0$. Then $\vert P\vert$ determines a morphism $f: S\to \Sigma\subset \PP^r$, with $\Sigma$ a minimal degree surface. Then $S$ is minimal, proving the assertion in this case too. \end{proof}

\begin{corollary}[Iitaka, \cite{Iitaka1}]  \label{cor:Iitaka1}
Let $(S,C)$  be a relatively minimal pair, with $S$ rational and $C$ irreducible.
If $\kappa(S,C)\geqslant 0$,
then  $C^2$ and $K_S^2$ are birational invariant of the pair $(S,C)$.
\end{corollary}

\begin{proof} The assertion follows by Proposition \ref{prop:Iitaka1} if either $m(S,C)>1$ or $m(S,C)=1$ and $(S,C)$ does not present the ruled case.
In the former case, however, we have $C+K_S\equiv \alpha L$
where $L$ is a pencil of rational curves. But the relative minimality hypothesis implies that there is no $(-1)$-curve $E$ such that $E\cdot L=0$.
Hence $S$ is a $\F_n$, and $C\equiv \alpha L-K_S$,
thus $C^2$ and $K_S^2$ do not depend on $n$.
\end{proof}

\subsection{Low Kodaira dimension}

\begin{proposition} [Kumar--Murthy, \cite{KumarMurthy};  Iitaka, \cite{Iitaka1}]   \label{prop:Iitaka2} Let $(S,C)$ be a relatively minimal pair with $S$ rational and $C$ irreducible. Assume $\kappa=\kappa(S,C)=0,1$. Then $(S,C)$ is one of the following:

\begin{enumerate}[(i)]

\item $(S,C)$ is the minimal resolution of singularities of $(\PP^2,D)$ with
$D\in\cL(3m; m^9)$,  $m\geqslant 1$ ($\kappa=0)$;

\item $(S,C)$ is the minimal resolution of singularities of  $(\F_n, D)$, $n=0,2$, with $D\in \cL(2m,(2+n)m; m^8)$, $m\geqslant 1$ ($\kappa=0$);

\item $(S,C)$ is the minimal resolution of singularities of $(\PP^2,D)$ with
$D\in\cL(3m; m^9,2)$,  $m\geqslant 1$ ($\kappa=0)$;

\item $(S,C)$ is the minimal resolution of singularities of  $(\F_n, D)$, $n=0,2$, with $D\in \cL(2m,(2+n)m; m^8,2)$, $m\geqslant 1$ ($\kappa=0$);

\item $(\F_n, D)$ as in case $(r)$ of Theorem \ref{thm1} with $m=1$ ($\kappa=1$).
\end{enumerate}
\end{proposition} 

\begin{proof} Let $m=m(S,C)$, and $mK_S+C=P+N$, with $P$ the nef and $N$ the negative part. One has $P^2=0$. If we are in the del Pezzo case, then we are in cases (i)--(iv). If we are in the ruled case, 
since $C+P$ is big, one has $m=1$ and we are in case (v). \end{proof}

\begin{remark} Cases (i)--(ii) in Proposition \ref{prop:Iitaka2} above are birational to minimal Cremona plane curves $D\in\cL(3m; m^9)$ which, for $m\geqslant 1$, give rise to \emph{Halphen's pencils} of elliptic curves. Cases in (iii)--(iv) are birational to minimal Cremona plane curves $D\in\cL(3m; m^9,2)$, which are nodal curves in  Halphen's pencils. 
\end{remark}

\subsection{Genus 0} Let $(S,C)$ be a relatively minimal pair with $S$ rational and $C$ irreducible and rational. If $C^2\geqslant -3$, then $h^0(S,\mathcal O_S(2K_S+C))=0$ and $(S,C)$ presents the line case. 

\begin{proposition}[Matsuda, \cite{Matsuda}] \label{prop:matsuda1} $(S,C)$ be a relatively minimal pair with $S$ rational and $C$ irreducible and rational. If $C^2=-4$ and $(S,C)$ does not present the line case,  then $\kappa(S,C)=0$ and $(S,C)$ is as in (iii) and (iv) of Proposition \ref{prop:Iitaka2}. The minimal Cremona type of $C$ is $(3m;m^9, 2)$,
i.e.\ $C$ is a nodal curve in a Halphen's pencil.
\end{proposition}

\begin{proof}  Set $K_S^2=k$ and $P=2K_S+C$ which is effective and nef. One has $P^2=4(k+1)\geqslant 0$, thus $k\geqslant -1$. Also $P\cdot K_S=2(k+1)$. Riemann--Roch says that $\dim(\vert P\vert )\geqslant k+1$. But since $C\cdot P=0$, one has $\dim(\vert P\vert )=0$, hence $k=-1$ and $P^2=P\cdot K_S=0$. Moreover $2(K_S+C)=P+C$. This implies $\kappa(S,C)\leqslant 1$. The assertion follows by Proposition \ref{prop:Iitaka2}. \end{proof}

\subsection{Genus 1} The following result are classical (see, e.g., \cite{Nagata1, Nagata2}).

\begin{proposition}\label{prop:g1} Let $(S,C)$ be a relatively minimal pair with $S$ rational and $C$ nef of genus $1$. Assume $C^2\geqslant 0$ with $C$ irreducible if $C^2=0$. Then $(S,C)$ is as in cases (i)--(ii) of Proposition \ref{prop:Iitaka2}, with $m=m(S,C)=1$ if and only if $C^2>0$. In particular  $(S,C)$ is birational to a pair in (i) of Proposition \ref{prop:Iitaka2}. 
\end{proposition}
\begin{proof} Set $K_S+C=P$ which is effective and nef. Since $P\cdot C=0$, if $C^2>0$, the Index Theorem yields $P=0$. If $C^2=0$, then $P^2=0$ and $\kappa(S,C)<2$, thus the assertion follows by Proposition \ref{prop:Iitaka2}. \end{proof}

\begin{corollary} \label{cro:genus1} Let $(S,\cL)$ with $S$ rational and $\cL$ an  linear system whose general curve $C$ is nef and irreducible if $C^2=0$, has arithmetic  genus 1 and $r=\dim(\cL)\geqslant 1$. Then either:
\begin{enumerate}[(i)]
\item $(S,\cL)$ is birational to $(\PP^2,\cL(3; 1^{9-r}))$, with $2\leqslant r\leqslant 9$,  or
\item $(S,\cL)$ is birational to $(\PP^2,\cL(3m; m^{9}))$, with $r=1$, or
\item $(S,\cL)$ is birational to  $(\F_n, \cL(2,2+n)$, $n=0,2$, or  to
$(\PP^2,\cL(4, 2^2))$, where the two base points may be infinitely near ($r=8$). 
\end{enumerate}
\end{corollary}

We add the following remarks:

\begin{corollary}\label{cor:planecubic} Let $(S,C)$ be a pair, with $S$ a rational surface and $C$ an irreducible curve of genus 1. Then $(S,C)$ is birational to $(\PP^2,B)$ with $B$ a smooth plane cubic if and only if $m(S,C)=1$.
\end{corollary}

\begin{proof} One implication is clear. As for the other, let $P=K_S+C$, which we may assume to be nef by contracting all $(-1)$-curves $E$ with $C\cdot E=0$. Since $P\cdot C=0$, only the del Pezzo case is possible, and the assertion follows.
\end{proof}

\begin{corollary}\label{cor:negell} Let $(S,C)$ be a relatively minimal pair, with $S$ a rational surface and $C$ an irreducible curve of genus 1 with $C^2<0$. Then $C^2<-1$. \end{corollary}

\begin{proof} Set $P=K_S+C$, which is effective and nef.  The case $m(S,C)=1$ is impossible since $P\cdot C=0$. Then $P'=2K_S+C$ is effective and nef.  Set $k=K^2$, Since $P^2=k+1$, we have $k\geqslant -1$. On the other hand $P'\cdot C=1$ and Riemann--Roch says that $\dim(\vert P'\vert)\geqslant k+1$. This implies $k=-1$. Bu then $(P')^2=-1$ a contradiction. \end{proof}

\subsection{De Franchis' Theorem} In this section we deal with the genus 2 case (for a classical reference, see \cite{Defranchis}; this is however affected, as well as all papers on the subject appeared before 1901, by the criticism raised by C. Segre to Noether's original proof of Noether--Castelnuovo's Theorem  \cite{Segre1}).

\begin{proposition}\label{prop:df1} Let $(S,\cL)$ be a pair with $S$ rational and $\cL$ a complete,  linear system of curves of  arithmetic genus $2$ with $r=\dim(\cL)\geqslant 1$ whose general curve $C$ is nef. Assume that there is no $(-1)$-curve $E$ such that $C\cdot E\leqslant 1$. Then $(S,\cL)$ is as follows:

\begin{enumerate}[(i)]
\item $(\F_n, \cL_n(2,3+n))$, with $0\leqslant n\leqslant 3$;
\item $(S,\cL)$ is the minimal desingularization of a pair of the form $(\PP^2, \cL(6;2^8))$;
\item $(S,\cL)$ is the minimal desingularization of  a pair of the form $(\PP^2, \cL(7;3, 2^{10}))$;
\item $(S,\cL)$ is the minimal desingularization of a pair of the form $(\PP^2, \cL(9;3^8,2^2))$;
\item $(S,\cL)$ is the minimal desingularization of a pair of the form $(\PP^2, \cL(13;5,4^9))$.
\end{enumerate}
\end{proposition}
\begin{proof} Set $K_S+C=P$, which is effective and nef. Actually $\vert P\vert$ is a pencil.
Since $P\cdot C=2$, then $P\neq 0$. 

If $P^2=0$, then $P\cdot K_S=-2$, and therefore $\vert P\vert$ is a pencil of curves of genus 0 and there is no $(-1)$-curve $E$ such that $E\cdot P=0$, namely $S=\F_n$, for some  $n\geqslant 0$.
Hence $m=m(S,C)=1$ and $\alpha=\alpha(S,C)=1$, and we are in case $(r)$ of Theorem \ref{thm1}
which leads to case (i).

If $P^2>0$, the Index Theorem yields $C^2\leqslant 4$.  If $P^2=1$  then $P\cdot K_S=1$. 
By contracting all $(-1)$-curves $E$ such that $P\cdot E=1$, we obtain a new pair $(S',P')$. This may let $\cL$ acquire base points of multiplicity 2. Then we may apply Proposition \ref{prop:g1} which tells us 
what  $(S',P')$ is, and we conclude that $(S,P)$ is the minimal desingularization of a pair of the form $(\PP^2, \cL(6;2^8))$ or  $(\F_n, \cL_n(4,2(2+n), 2^7)$, with $n=0,2$. Thus
we are  in case (ii).  In particular this applies when $C^2=4$. 

So we are left with the cases $P^2\geqslant 2$, and the cases $C^2=1, P^2=4$ and  $C^2=P^2=2$ can be excluded by the Index Theorem.  

Consider first the case $r>1$. Set $K_S^2=k$.  If $-K_S$ is not effective, then $\cL$ cuts on $P$ a linear series of dimension 2 and degree 2. hence $P$ is rational, hence $m(S.C)=1$ and we are in the big case since $P^2>0$. This contradicts $\alpha=1$. So $-K_S>0$. Then  $2\leqslant P^2=P\cdot (K_S+C)\leqslant P\cdot C=2$. Hence $P^2=2$ and  then $P\cdot K_S=0$.  Moreover $C^2\leqslant 1$, $C\cdot K_S=2-C^2$, so that $0=P\cdot K_S=k+2-C^2$. In conclusion $C\cdot K_S=-k+K_S\cdot P=-k=2-C^2\geqslant 0$, a contradiction. 

So we may assume $r=1$.  Then $C^2\leqslant 2$, since $1=r=h^0(S,\mathcal O_S(C))-1=h^0(C,\mathcal O_C(C))\geqslant C^2-1$. 

Let $C^2=c$, $C\cdot K_S=2-c$, with $0\leqslant c\leqslant 1$.
Then $P^2=k+4-c$, $P\cdot K_S=k+2-c$, hence $p_a(P)=k+4-c=P^2$. Set $P'=P+K_S=C+2K_S$. Then $\dim(\vert P'\vert)=P^2-1$ and $C\cdot P'=4-c$. 
Therefore $P^2 \leqslant 3-c$, since $P'-C\equiv 2K_S$ is not effective. 

So we are left with the cases $C^2=0, 2\leqslant P^2\leqslant 3$, $C^2=1, P^2=2$.

If $C^2=0, P^2=2$, then $k=-2$,  $P'$ is still nef and $(P')^2=0$. Moreover $P\cdot P'=2$ implies that $\vert P'\vert$ is a pencil of rational curves, so that we are in the ruled case with $m=2, \alpha=1$ and $(S,\cL)$ is the minimal desingularization of  a pair of the form $(\F_n, \cL_n(4,5+2n; 2^{10}))$, with $0\leqslant n\leqslant 2$ and we are in case (iii).

If $C^2=1, P^2=2$, then $k=-1$, $P'$ is nef and $(P')^2=1$, $P'\cdot K_S=-1$. Hence $\vert P'\vert$ is a pencil of curves of arithmetic genus 1. By Corollary \ref{cro:genus1} we are in case (iv).

If $C^2=0, P^2=3$,   then $k=-1$, $P'$ is nef and $(P')^2=4$, $P'\cdot K_S=0$.  Now let us contract all $(-1)$-curves $E$ such that $E\cdot P'=0$, producing a new pair $(S_1,P_1)$ on which $\cL$ may acquire base points of multiplicity 2.  Set $P'_1=K_{S_1}+P_1$, which is nef,  and $\dim (\vert P'_1\vert)=2$. We have $P_1^2=4+h, K_{S_1}^2=h-1$, where $h$ is the number of the $(-1)$--cycles on $S$ such that $E\cdot P'=0$. Then $(P_1')^2=3, P_1'\cdot K_{S_1}=-1$. The analysis of cases (i)--(ii) above tells us that $(S_1,\vert P'_1\vert)$ arises by minimally resolving the base points of 
either  $\PP^2, \cL(4;2,1^9))$ or of $\PP^2, \cL(6;2^8,1))$, in both cases being $h=0$. Note that the cases
$\PP^2, \cL(5;3,2,1^9))$ and $\PP^2, \cL(6;4, 2^2,1^9))$ are impossible because $k=-1$. On the other hand, the pair $\PP^2, \cL(6;2^8,1))$ does not come by adjunction from a system of curves of genus 2, wheres the pair $\PP^2, \cL(4;2,1^9))$ does, leading to case (v). \end{proof}

\begin{remark} Note that in case (i) of  Proposition \ref{prop:df1}, the pair $(S,\cL)$ is also the minimal desingularization of one of the pairs:
\begin{itemize}
\item $\cL(4,2)$ for  $n=1$;  
\item $\cL(5,3, 2)$ for  $n=0,2$, where for $n=2$ the two base points are infinitely near; 
\item $\cL(6;(4,[2^2]))$ for  $n=3$.
\end{itemize}
These linear systems, as well as the others appearing in (ii)--(v) of Proposition \ref{prop:df1}, are Cremona minimal, because of Corollary \ref{cor: jung}. 

In all cases, but (iii)--(v), of Propositon \ref{prop:df1} one has $h^1(S,\mathcal O_S(C))=0$, i.e.\ the system $\cL$ is non--special. So the existence of these systems is not under question, since one may choose the base points to be general. The question is different for the cases (iii)--(v), in which we have 
$h^1(S,\mathcal O_S(C))=1$ in case (iv) and $h^1(S,\mathcal O_S(C))=2$ in the remaining two cases.  The existence of these systems is discussed in \cite{Defranchis}. \end{remark}

\subsection{Cremona equivalence to smooth curves} In \cite{Coolidge}, pp. 399--401, Coolidge considers the question of characterizing those irreducible plane curves $B$ which are Cremona equivalent to a smooth curve. His answers are rather complicated.
A simple characterization can be given using Theorem \ref{thm1}. 

\begin{proposition}\label{prof:appl} Let $B$ be an irreducible plane curve of genus $g$ which is not Cremona equivalent to a line. Then $B$ is Cremona equivalent to a smooth plane curve of degree $d\equiv 0,1$ modulo 3,  if and only if  $m(S,C)=[\frac d3]$ and respectively $\alpha(S,C)=0,1$, where  $(S,C)$ is  the minimal  desingularization of $(\PP^2,B)$, and $g=\binom{d-1}{2}$.
The same assertion holds for $d\equiv  2$ modulo 3,  if $\alpha(S,C)=5$ as soon as either $d\geqslant 29$ or the pair $(S,C)$ does not present the ruled case. 
\end{proposition}

\begin{proof} We prove only the non--trivial implication. Let $d\equiv 0$ modulo 3 and assume $m=m(S,C)=[\frac d3]$ and $\alpha=0$. Then, computing the genera of all pairs in Theorem \ref{thm1} for these values of the invariants, one sees that, in order to have $g=\binom{d-1}{2}$, only case $(dp_1)$ can occur, in which case we have the assertion. The proof of the case $d\equiv 1$ modulo 3 goes in the same way. In case $d\equiv 2$ modulo 3 the same proof works as son as $m>8$. If $m\leqslant 8$, one may have cases in which the genus of the curves in the list Theorem \ref{thm1} is larger than $\binom{d-1}{2}$ only in the ruled case, for $m\leqslant 8$. The assertion follows. \end{proof}

Coolidge also gives conditions under which a plane curve is Cremona equivalent to  a curve with only double points. This can be also treated as in the previous proposition. We do not dwell on tis here.

\appendix

\section{A proof of the Noether-Castelnuovo Theorem via simplicity}\label{app:proof}

Here we give a proof of Noether-Castelnuovo Theorem \ref{thm:NC} by induction
on the simplicity of the homaloidal net $\cL$
defining a Cremona transformation $\phi=\phi_\cL\colon\PP^2\rto\PP^2$,
namely we will show that, if $\phi$ is not quadratic or linear,
i.e.\ if its simplicity is larger than $(1,2,0)$,
then there is a quadratic trasformation $\gamma\colon\PP^2\rto\PP^2$
such that $\gamma\circ\phi$ is simpler than $\phi$.

The notion of simplicity is essentially due to Alexander  \cite{Alex} (cf.\ also Segre \cite{Segre}  and Chisini \cite{Chisini}).

Write $\cL=\cL(\delta;\alpha_0,\ldots,\alpha_r)$, where
$\alpha_0\geqslant\cdots\geqslant\alpha_r$ and let $p_i$, $i=0,\ldots,r$, be the point
of multiplicity $\alpha_i$ of $\cL$.
Let $(k_\phi,h_\phi,s_\phi)$ be the simplicity of $\phi$,
which is that of $\cL$, as defined in \eqref{eq:simplicity}.
If $\delta=2$, there is nothing to prove. Suppose then $\delta>2$.
Note that $k_\phi=\delta-\alpha_0\geqslant1$.

By subtracting $\alpha_0$ times the latter equation
from the former one in Eqs.\ \eqref{eq:invariantsLambda},
one gets
\[
(\delta-1)(\delta-3\alpha_0 + 1) = \sum_{i=1}^r\alpha_i(\alpha_i-\alpha_0)\leqslant 0,
\]
hence $\delta<3\alpha_0$, or equivalently $\alpha_0>k_\phi/2$, therefore $h_\phi\geqslant0$.
Similarly, by subtracting $m=k_\phi/2$ times the latter equation
from the former one in Eqs.\ \eqref{eq:invariantsLambda}, one gets
\[
(\delta-1)(\delta-3m+1)=\delta(\delta-3m)+3m-1
 = \alpha_0(\alpha_0-m)+\sum_{i=1}^r\alpha_i(\alpha_i-m),
\]
which, since $\delta-3m=\alpha_0-m$ and $m=k_\phi/2\geqslant1/2$, implies that
\[
2m(\alpha_0-m)=(\delta-\alpha_0)(\alpha_0-m)
  <  \sum_{i=1}^r\alpha_i(\alpha_i-m)
\leqslant \sum_{i=1}^h\alpha_i(\alpha_i-m)
\leqslant \sum_{i=1}^h 2m(\alpha_i-m),
\]
where the last inequality follows from the fact that
$\alpha_i\leqslant \delta-\alpha_0=2m$, for each $i>0$.
Thus
\begin{equation}\label{eq:a0-m}
\alpha_0-m \leqslant \sum_{i=1}^h (\alpha_i-m),
\end{equation}
which implies that $h_\phi\geqslant2$ and $p_1,\ldots,p_{h_\phi}$ cannot be all proximate to $p_0$,
in particular they cannot be all infinitely near to $p_0$ of order 1.
Note that, for $0<i<j\leqslant h_\phi$, the points $p_0,p_i,p_j$ are not aligned,
otherwise $\alpha_0+\alpha_i+\alpha_j>\delta$
and $\cL$ is reducible, a contradiction.

Suppose first that there are two points
among $p_1,\ldots,p_{h_\phi}$, say $p_{i}$ and $p_{j}$,
such that the quadratic transformation $\gamma$ centered at $p_0, p_i, p_j$ exists.
Then $\psi=\gamma\circ\phi$ is a Cremona transformation of degree
$\delta-\epsilon$, where $\epsilon=\alpha_i+\alpha_j-2m>0$,
and it is given by a homaloidal net $\Lambda$ having multiplicity $\alpha_0-\epsilon<\alpha_0$,
$\alpha_j-m\leqslant m$, $\alpha_i-m\leqslant m$ respectively
at the points (corresponding to) $p_0, p_i, p_j$.
Either $p_0$ is still a maximal multiplicity point of $\Lambda$,
hence $k_{\psi}=\delta-\alpha_0=k_\phi$ and $h_{\psi}=h_\phi-2$
(because $p_i$ and $p_j$ are now of multiplicity $\leqslant m=k_\psi/2$),
or the net $\Lambda$ has maximal multiplicity $\mu>\alpha_0-\epsilon$,
therefore $k_\psi<\delta-\alpha_0=k_\phi$.
In both cases, $\psi$ is simpler than $\phi$.

If there is no such quadratic transformation $\gamma$,
it means that $p_i\infnear p_0$, for each $i=1,\ldots,h$,
and, by \eqref{eq:a0-m}, there are $p_i, p_j$ such that $p_j\infnear[1] p_i\infnear[1] p_0$
and $p_j\satel p_0$ (if $p_j\notsatel p_0$, then $\gamma$ exists).
Chosen a general point $q\in\PP^2$, consider 
the quadratic transformation $\gamma'$ centered at $p_0,p_i,q$
and set $\psi'=\gamma'\circ\phi$.
Then $\psi'$ has degree $\delta+\epsilon'$, where $0\leqslant \epsilon'=2m-\alpha_i < m$
and it is given by a homaloidal net $\Lambda'$ having
multiplicity $\alpha_0+\epsilon'\geqslant\alpha_0$,
$2m$, $\epsilon'<m$ respectively
at the points (corresponding to) $p_0, p_i, q$.
Therefore $k_{\psi'}=\delta-\alpha_0=k_\phi$ and $h_{\psi'}=h_{\phi}$.
But the base point of $\Lambda'$ corresponding to $p_j$
is now infinitely near to $p_0$ of order 1, hence it is not satellite anymore
and $s_{\psi'}=s_\phi-1$.
Therefore, $\psi'$ is simpler than $\phi$.

Repeating this argument, we eventually get a Cremona transformation
with simplicity $(1,2,0)$, which is a quadratic one, and the proof is concluded.
\qed

\end{document}